\documentclass[10pt,a4paper,reqno]{amsart}
\usepackage{amsmath}
\usepackage{latexsym}
\usepackage{amssymb} 
\usepackage{mathrsfs}	% \mathscr
\usepackage{amscd}
\usepackage{amsfonts}
\usepackage{color}
\usepackage[all]{xy}
\usepackage{enumitem}
\usepackage{amsrefs}	% for BibTeX

\newif\ifLabelPrint 
\LabelPrinttrue	% Line 102--126
\newtheorem{thm}{Theorem}[section]
\newtheorem{prop}[thm]{Proposition}
\newtheorem{cor}[thm]{Corollary}
\newtheorem{lem}[thm]{Lemma}

\theoremstyle{definition}
\newtheorem{defi}[thm]{Definition}
\newtheorem{rmk}[thm]{Remark}
\newtheorem{nota}[thm]{Notation}

\newtheorem*{acknowledgement}{Acknowledgement}

\let\op\operatorname			% use \op{...} to write down operators
\newcommand*{\Z}{\mathbb{Z}}	% Set of integers
\newcommand*{\C}{\mathbb{C}}	% Set of complex numbers
\renewcommand*{\P}{\mathbb{P}}	% Projectivization
\newcommand*{\D}{\op{D}^{\mathsf{b}}}	% Bounded derived category
\newcommand*{\Hom}{\op{Hom}}	% Hom-Group
\newcommand*{\Ext}{\op{Ext}}	% Ext-Group
\newcommand*{\varHom}{\mathit{\mathcal H\hskip-1pt{}om}}	% Sheaf-Hom
\newcommand*{\Pic}{\op{Pic}}	% Picard Group
\newcommand*{\Dotimes}{ \mathbin{\stackrel{L}{\otimes}}}	% derived tensor product
\newcommand*{\rank}{\op{rank}} % Rank
\makeatletter\@addtoreset{equation}{section}\makeatother
\setlist[enumerate,1]{leftmargin=25pt,label={\normalfont(\arabic{enumi})},topsep=0pt}
\setlist[itemize,1]{leftmargin=25pt,topsep=0pt}
\begin{document}

\title[Ulrich bundles on intersections of two 4-dimensional quadrics]{Ulrich bundles on \\ intersections of two 4-dimensional quadrics}

\author{Yonghwa Cho}
\address{KAIST, Daehak-ro 291, Yuseong-gu, Daejeon 34141, Republic of Korea}
\email{yonghwa.cho@kaist.ac.kr}

\author{Yeongrak Kim}
\address{Max-Planck Institute for Mathematics, Vivatsgasse 7, Bonn 53111, Germany}
\email{yeongrakkim@mpim-bonn.mpg.de}

\author{Kyoung-Seog Lee}
\address{Center for Geometry and Physics, Institute for Basic Science (IBS), Pohang 37673, Republic of Korea}
\email{kyoungseog02@gmail.com}

\maketitle

%
%
%	To print label marker for editing purposes (for using amsmath package with reqno)
%\ifLabelPrint%
%	\makeatletter%
%	\let\ltx@label\label%
%	\let\oldlabel\ltx@label%
%	\def\ltx@label#1{%
%		\relax\oldlabel{#1}%
%		\ifmmode\let\lap\llap%
%		\else%
%			\ifinalign@\let\lap\llap%
%			\else\let\lap\rlap%
%			\fi
%		\fi%
%		\raisebox{0.6\baselineskip}[0pt][0pt]{\lap{\color{blue}\small\verb~\detokenize{#1}~}}%
%	}%
%	\let\label\ltx@label\makeatother%%	
%
%	Discard this if you do not use BibTeX...
%	\let\oldbibitem\bibitem%
%	\def\bibitem#1{%
%		\oldbibitem{#1}%
%		\raisebox{0.6\baselineskip}[0pt][0pt]{\rlap{\color{blue}\small\verb~\detokenize{#1}~}}%
%	}%
%:
%\fi%
%:
%\setcounter{tocdepth}{1}\tableofcontents

\begin{abstract}
	In this paper, we investigate the existence of Ulrich bundles on a smooth complete intersection of two $4$-dimensional quadrics in $\P^5$ by two completely different methods. First, we find good ACM curves and use Serre correspondence in order to construct Ulrich bundles, which is analogous to the construction on a cubic threefold by Casanellas-Hartshorne-Geiss-Schreyer. Next, we use Bondal-Orlov's semiorthogonal decomposition of the derived category of coherent sheaves to analyze Ulrich bundles. Using these methods, we prove that any smooth intersection of two 4-dimensional quadrics in $\P^5$ carries an Ulrich bundle of rank $r$ for every $r \ge 2$. Moreover, we provide a description of the moduli space of stable Ulrich bundles.
\end{abstract}

\section{Introduction}
Let $\P^n$ be the $n$-dimensional projective space over the field of complex numbers $\C$. A famous theorem by Horrocks states that a vector bundle $\mathcal E$ on $\P^n$ splits as the direct sum of line bundles if and only if $\mathcal E$ has no intermediate cohomology, \emph{i.e.}, $h^i(\mathcal E(j))= 0$ for all $0<i<n$ and $j \in \Z$. It is natural to ask the algebro-geometric meaning of these vanishing conditions for the other varieties. Let  $X \subset \P^N$ be an $n$-dimensional smooth projective variety with a fixed polarization $\mathcal O_X(1) = \mathcal O_{\P^N}(1)|_{X}$. We call that a vector bundle $\mathcal E$ on $X$ is \emph{ACM (arithmetically Cohen-Macaulay)} if $\mathcal E$ has no intermediate cohomology with respect to the given polarization $\mathcal O_X (1)$. Roughly speaking, the presence of nontrivial ACM bundles measures how $X$ is apart from the projective space $\P^n$. Due to their interesting properties, ACM bundles have played a significant role in the study of vector bundles. 

In commutative algebra, ACM bundles correspond to \emph{MCM (maximal Cohen-Macaulay)} modules which are Cohen-Macaulay modules achieving the maximal dimension. A particularly interesting case happens when the minimal free resolution of an MCM module becomes completely linear. Such an MCM module has the maximal possible number of minimal generators which are concentrated on a single degree \cite{Ulrich1984}. Eisenbud and Schreyer made a comprehensive study on the geometric analogue of these linear MCM modules, and named them Ulrich sheaves \cite{EisenbudSchreyer2003}. Thanks to foundational works by Beauville \cite{Beauville2000} and Eisenbud-Schreyer \cite{EisenbudSchreyer2003}, Ulrich sheaves provide a number of fruitful applications; for example, linear determinantal representations of hypersurfaces, matrix factorizations by linear matrices, the cone of cohomology tables, and Cayley-Chow forms. Eisenbud and Schreyer conjectured that every projective variety carries an Ulrich sheaf \cite{EisenbudSchreyer2003}, and verified it for a few simple cases. The conjecture is still wildly open even for smooth surfaces. In very recent years, there were several progresses on the conjecture for surfaces; for instance, general K3 surfaces \cite{AproduFarkasOrtega}, abelian surfaces \cite{Beauville2015:Abelian}, and nonspecial surfaces of $p_g = q = 0$ \cite{Casnati}. 

Much less is known for ACM and Ulrich bundles on threefolds. On a smooth quadric $Q^3 \subset \P^4$, there is only one nontrivial indecomposable ACM bundle, namely, the spinor bundle \cite{BuchweitzEisenbudHerzog}. Arrondo and Costa studied ACM bundles on Fano 3-folds of index 2 of degree $d=3, 4, 5$ \cite{ArrondoCosta2000}. Madonna studied splitting criteria for rank 2 vector bundles on hypersurfaces in $\P^4$ \cite{Madonna1998:Splitting}. He also classified all the possible Chern classes of rank 2 ACM bundles on prime Fano 3-folds and complete intersection Calabi-Yau 3-folds \cite{Madonna2002:ACMprimeFano}. Their results expected the existence of Ulrich bundles on 3-folds of small degree, however, constructions were not complete except for a very few cases. On the other hand, Beauville showed that a general hypersurface of degree $\le 5$ in $\P^4$ is linearly Pfaffian. In other words, such a hypersurface carries a rank 2 Ulrich bundle \cite{Beauville2000}. He also checked that every Fano 3-fold of index 2 carries a rank 2 Ulrich bundle \cite{Beauville2016:IntroductionUlrich}. In particular, a general smooth cubic 3-fold carries Ulrich bundles of rank $r$ for every $r \ge 2$, proved first by Casanellas, Hartshorne, Geiss, and Schreyer \cite{CasanellasHartshorne:Ulrich}. Recently, Lahoz, Macr{\`i}, and Stellari extended this result to every smooth cubic 3-fold using the derived category of coherent sheaves and also described the moduli space of stable Ulrich bundles \cite{LahozMacriStellari}. 

It is quite natural to ask for the next case, a del Pezzo threefold $X=Q_0^4 \cap Q_{\infty}^4$ of degree four which is the complete intersection of two quadric 4-folds. Indeed, $X$ is very attractive since there are several ways to understand vector bundles on $X$. Since $X$ is a 3-fold, we may construct vector bundles on $X$ by observing curves lying on $X$ via Serre correspondence. On the other hand, it is also well-known that the geometry of the intersection of $2$ even dimensional quadrics is closely related to a hyperelliptic curve. Bondal and Orlov showed that the derived category of coherent sheaves on the intersection of $2$ even dimensional quadrics has a semiorthogonal decomposition whose components are the derived category of the hyperelliptic curve associated to the $2$ given quadrics and the exceptional collection \cite{BondalOrlov:SODforAlgVar}. Recently, there were several attempts to understand vector bundles on a variety using the semiorthogonal decomposition of its derived category. For instance, Kuznetsov studied instanton bundles on some index 2 Fano 3-folds via  semiorthogonal decompositions \cite{Kuznetsov:Instanton}. Lahoz, Macr{\`i}, and Stellari studied ACM bundles on cubic 3-folds and 4-folds via semiorthogonal decomposition \cite{LahozMacriStellari, LahozMacristellari:4folds}. Therefore, it is reasonable to apply the semiorthogonal decomposition to understand vector bundles on the intersection of two even dimensional quadrics.

Being motivated by earlier works mentioned above, we investigate the existence and the moduli space of Ulrich bundles on the intersection of two 4-dimensional quadrics by two completely different methods: classical Serre corrseponence and Bondal-Orlov theorem. The main result is the following theorem:

	\begin{thm}[see Theorem~\ref{thm:StableUlrichBundlesClassical} and \ref{thm: Main thm}]
		The moduli space of stable Ulrich bundles of rank $r\geq 2$ on $X = Q_0^4 \cap Q_\infty^4$ is isomorphic to a nonempty open subscheme of $\mathcal U_C^{\sf s}(r,2r)$, where $\mathcal U_C^{\sf s}(r,2r)$ is the moduli space of stable vector bundles of rank $r$ and degree $2r$ on a curve $C$ of genus $2$.
	\end{thm}

Our approach using Serre correspondence closely follows the works of Arrondo and Costa \cite{ArrondoCosta2000} and of Casanellas, Hartshorne, Geiss, and Schreyer \cite{CasanellasHartshorne:Ulrich}, and our approach using derived categories is strongly influenced by the works of Kuznetsov \cite{Kuznetsov:Instanton} and of Lahoz, Macr{\`i}, and Stellari \cite{LahozMacriStellari, LahozMacristellari:4folds}. The structure of this paper is as follows. In Section \ref{section:Preliminary}, we recall a few useful facts related to ACM and Ulrich bundles. In Section \ref{section:Serre}, we construct Ulrich bundles of any rank $r \ge 2$ on a general intersection of two quadric 4-folds $X=Q_0^4 \cap Q_{\infty}^4$ using Serre correspondence and \emph{Macaulay2}. In Section \ref{section:derived}, we prove the existence of Ulrich bundles of any rank $r \ge 2$ on a smooth complete intersection of two quadric 4-folds $X=Q_0^4 \cap Q_{\infty}^4$ using Bondal-Orlov theorem. We also analyze the moduli of stable Ulrich bundles of rank $r$ on $X$ and provide a description in terms of vector bundles on $C$.

%:
\section{Preliminaries on ACM and Ulrich bundles}\label{section:Preliminary}%

In this section, we briefly review the definition of ACM and Ulrich bundles and their basic properties. 

\begin{defi}\label{defi:ACMUlrich}
Let $X \subset \P^N$ be an $n$-dimensional smooth projective variety embedded by a very ample line bundle $\mathcal O_X(1)$. 
\begin{enumerate}
\item A coherent sheaf $\mathcal E$ on $X$ is \emph{ACM} if $H^i (\mathcal E(j))=0$ for all $0<i<n$ and $j \in \Z$.
\item An ACM sheaf $\mathcal E$ on $X$ is \emph{Ulrich} if $H^0(\mathcal E(-1))=0$ and $h^0(\mathcal E)=\deg (X)  \rank(\mathcal E)$.
\end{enumerate}
\end{defi}

\begin{rmk}
Since the underlying space $X$ is smooth, $\mathcal E$ being $ACM$ implies that $\mathcal E$ is locally free. Hence it is natural to call ACM (Ulrich) bundles for the objects occurring in the above definition.
\end{rmk}

We recall the following proposition by Eisenbud and Schreyer. We refer to \cite{Beauville2016:IntroductionUlrich, EisenbudSchreyer2003} for more details.
\begin{prop}[{\cite[Theorem~1]{Beauville2016:IntroductionUlrich},\ \cite[Proposition 2.1]{EisenbudSchreyer2003}}]\label{prop: Ulrich Equiv conditions}
Let $X \subset \P^N$ and $\mathcal E$ as above. The following are equivalent:
\begin{enumerate}
\item $\mathcal E$ is Ulrich;
\item $H^i (\mathcal E(-i))=0$ for all $i>0$ and $H^j (\mathcal E(-j-1))=0$ for $j<n$.
\item $H^i (\mathcal E(-j))=0$ for all $i$ and $1 \le j \le n$. \label{item: Ulrich as Acyclic condition}
\item For some (all) finite linear projections $\pi : X \to \P^n$, the sheaf $\pi_{*} \mathcal E$ is isomorphic to the trivial sheaf $\mathcal O_{\P^n}^{\oplus t}$ for some $t$.
\item The section module $M:=\oplus_j H^0 (\mathcal E(j))$ is a linear MCM module, that is,
the minimal $S=\C[x_0, \ldots, x_N]$-free resolution of $M$ 
\[
\mathbf{F} : 0 \to F_{N-n} \to \cdots \to F_1 \to F_0 \to M \to 0
\]
is linear in the sense that $F_i$ is generated in degree $i$ for every $i$.
\end{enumerate}
\end{prop}

In particular, by Serre duality, we immediately have the following proposition as a consequence:
\begin{prop}\label{prop:DualOfACMUlrich}
Let $X^n \subset \P^N$ be as above, and let $H:= \mathcal O_X(1)$ be a very ample line bundle.
\begin{enumerate}
\item If $\mathcal E$ is an ACM bundle on $X$, then $\mathcal E^* (K_X)$ is also an ACM bundle.
\item When $X$ is subcanonical, that is, $K_X = \mathcal O_X (k)$ for some $k \in \Z$,  $\mathcal E$ is ACM if and only if $\mathcal E^*$ is ACM.
\item If $\mathcal E$ is an Ulrich bundle on $X$, then $\mathcal  E^* (K_X + (n+1)H)$ is an Ulrich bundle.
\end{enumerate}
\end{prop}

The following proposition about the stability is very useful in later sections.
\begin{prop}[{\cite[Theorem 2.9]{CasanellasHartshorne:Ulrich}}]\label{prop: semistability of Ulrich}
Let $X$ be a smooth projective variety, and let $\mathcal E$ be an Ulrich bundle on $X$. Then
\begin{enumerate}
\item $\mathcal E$ is semistable and $\mu$-semistable.
\item If $0 \to \mathcal E' \to \mathcal E \to \mathcal E'' \to 0$ is an exact sequence of coherent sheaves with $\mathcal E''$ torsion-free, and $\mu(\mathcal E') = \mu(\mathcal E)$, then both $\mathcal E'$ and $\mathcal E''$ are Ulrich.
\item If $\mathcal E$ is stable, then it is also $\mu$-stable.
\end{enumerate}
\end{prop}
%:

%:
%:
%:

\section{Geometric approach via Serre correspondence}\label{section:Serre}
In this section, we show the existence of Ulrich bundles using Serre correspondence.

\subsection{Serre correspondence}
We briefly recall Serre correspondence which enables us to construct a vector bundle as an extension from a codimension 2 subscheme. To obtain a vector bundle, such a subscheme has to satisfy certain generating conditions. For instance, it is well-known that a 0-dimensional subscheme on a smooth surface should satisfy Cayley-Bacharach condition to provide a locally free extension. For higher dimensional cases, the situation gets much more complicated. For example, a curve in $\P^3$ occurs as the zero locus of a rank 2 vector bundle on $\P^3$ if and only if it is a local complete intersection and subcanonical \cite{Hartshorne1978}. It is clear that not all curves come from vector bundles. When it happens, we cannot construct a vector bundle as an extension. However, still in many cases, it is a powerful tool providing constructions of vector bundles. We refer to \cite{Arrondo2007:SerreCorrespondence} for the proof and more details.

\begin{thm}[Serre correspondence] \label{thm:SerreCorrespondence}
Let $X$ be a smooth variety and let $Y \subset X$ be a local complete intersection subscheme of codimension $2$ in $X$. Let $\mathcal N$ be the normal bundle of $Y$ in $X$ and let $\mathcal L$ be a line bundle on $X$ such that $H^2(\mathcal L^*) = 0$. Assume that $(\wedge^2 \mathcal N \otimes \mathcal L^* )|_Y$ has $(r-1)$ generating global sections $s_1, \ldots, s_{r-1}$. Then there is a rank $r$ vector bundle $\mathcal E$ as an extension
\[
0 \to \mathcal O_X^{r-1} \xrightarrow{(\alpha_1, \ldots, \alpha_{r-1})} \mathcal E \longrightarrow \mathcal I_{Y/X} (\mathcal L) \to 0
\]
such that the dependency locus of $(r-1)$ global sections $\alpha_1, \ldots, \alpha_{r-1}$ of $\mathcal E$ is $Y$ with $\sum_{i=1}^{r-1} s_i \alpha_{i}|_Y = 0$. Moreover, if $H^1(\mathcal L^*) = 0$, such an $\mathcal E$ is unique up to isomorphism.
\end{thm}

\subsection{ACM bundles of rank 2 via Serre correspondence}\label{sec: Serre construction}
	From now on, let $Q_0, Q_\infty$ be two smooth quadric hypersurfaces in $\P^5$ meeting transversally and let $X = Q_0 \cap Q_\infty \subset \P^5$ be a smooth Fano 3-fold of degree 4 and index 2, \emph{i.e.}, $\omega_X = \mathcal O_X(-2)$.
	
	Let $[H_X]$, $[L_X]$, $[P_X]$ be the class of a hyperplane section, a line, and a point in $X$ respectively. Then, 
	\begin{equation}\label{eq: Cohomologies of X}
		H^2(X,\Z) \simeq \Z\cdot [H_X],\quad H^4(X,\Z) \simeq \Z \cdot [L_X],\quad \text{and}\quad H^6(X,\Z) \simeq \Z \cdot [P_X]. 
	\end{equation}
	The ring structure is given as follows: $H_X^2 = 4L_X$, $H_X \cdot L_X = P_X$.
	For a vector bundle $\mathcal F$ on $X$, we define its slope $\mu$ with respect to $H$ by
	\[
		\mu_H(\mathcal F) := \frac{\deg_H \mathcal F}{\op{rank} \mathcal F}
	\]
	By virtue of (\ref{eq: Cohomologies of X}), we fix our convention as follows.
	\begin{nota}
		Via the isomorphisms $\Z \cdot [H_X] \simeq \Z$, $\Z \cdot [L_X] \simeq \Z$, and $\Z \cdot [P_X] \simeq \Z$, we may regard $c_{i}(\mathcal F)$ as an integer, by omitting the cyclic generators of $H^{2i}(X,\Z)$. Under this convention, one can easily see that
		\[
			\mu_H(\mathcal F) = \frac{c_1(\mathcal F) \deg X}{\op{rank} \mathcal F} = 4 \cdot \frac{c_1(\mathcal F)}{\op{rank} \mathcal F}
		\]
		We also omit the redundant coefficient $4$ in the formula and redefine the slope of $\mathcal F$ as follows:
		\[
			\mu(\mathcal F):= \frac{c_1(\mathcal F)}{\op{rank}\mathcal F}.
		\]
	\end{nota}
%:
	The following proposition is useful in later sections.
%:	
	\begin{prop}[{\cite[Proposition~1.2.7]{HuybrechtsLehn:ModuliofSheaves}}]\label{prop: Stability vanishing}
		Let $\mathcal E$ and $\mathcal E'$ be $\mu$-stable bundles with $\mu(\mathcal E) > \mu(\mathcal E')$. Then $\Hom(\mathcal E,\mathcal E') = 0$.
	\end{prop}
%:
%:

%:
	Applying Proposition~\ref{prop:DualOfACMUlrich} to $X = Q_0 \cap Q_\infty$, we get the following:
	\begin{prop}
		Let $\mathcal E$ be an Ulrich bundle of rank $r$ on $X = Q_0 \cap Q_\infty$. Then,
		\begin{enumerate}
			\item $\mu(\mathcal E)=1$, and
			\item $\mathcal E^*(2)$ is an Ulrich bundle.
		\end{enumerate}
	\end{prop}
%:

In \cite{ArrondoCosta2000}, Arrondo and Costa made a comprehensive study of ACM bundles on $X$ extending \cite{SzurekWisniewski1993}. They classified the possible Chern classes for ACM bundles under a mild assumption. In particular, they classified all the rank 2 ACM bundles on $X$.

\begin{thm}[{\cite[Theorem 3.4]{ArrondoCosta2000}}]\label{thm: Arrondo-Costa ACM rk 2}
An indecomposable rank $2$ ACM vector bundle on $X$ is a twist of one of the following;
\begin{enumerate}
\item A line type: a semistable vector bundle $\mathcal E_{l}$ fitting in an exact sequnce
\[
0 \to \mathcal O_X \to \mathcal E_{l} \to \mathcal I_{l} \to 0
\]
where ${l} \subset X$ is a line contained in $X$;
\item A conic type: a stable vector bundle $\mathcal E_\lambda$ fitting in an exact sequence
\[
0 \to \mathcal O_X \to \mathcal E_\lambda \to \mathcal I_\lambda (1) \to 0
\]
where $\lambda \subset X$ is a conic contained in $X$;
\item An elliptic curve type: a stable vector bundle $\mathcal E_e$ fitting in an exact sequence
\[
0 \to \mathcal O_X \to \mathcal E_e \to \mathcal I_e (2) \to 0
\]
where $e \subset X$ is an elliptic curve of degree $6$. 
\end{enumerate}
\end{thm}

It is classically well-known that the Fano scheme $F(X)$ of lines $l \subset X$ is isomorphic to the Jacobian $J(C)$ of the hyperelliptic curve $C$ of genus 2 associated to $X$\,(see \cite[Theorem~5]{NarasimhanRamanan:ModuliofVectBdl}, \cite[Theorem~2]{Newstead:StableBundlesofRank2OddDeg} or \cite{Reid:PhD}). Since $\mathcal E_l$ has the unique global section up to constants, the space also coincides with the space of line type ACM bundles. 

Conic type ACM bundles are also well understood as in the following way. Given a conic $\lambda \subset X$, note that there is only one quadric $Q \in \mathfrak d := \lvert Q_0 + t Q_\infty \rvert_{t \in \P^1}$ in a pencil containing the plane $\Lambda = \left< \lambda \right>$. It is clear that $\Lambda \cap X = \lambda$. Since $Q$ is a 4-dimensional quadric, there is a spinor bundle whose global sections sweep out a family of planes in $Q$ containing $\Lambda$. The bundle $\mathcal E_\lambda$ is the restriction of this spinor bundle. Hence, the moduli of conic type ACM bundle can be naturally identified with the space of spinor bundles associated to the pencil $\mathfrak d$. 

The last case is particularly interesting. When $e \subset X \subset \P^5$ is an elliptic normal curve of degree 6, we have $h^0(\mathcal I_e(1))=0$ and $h^0(\mathcal I_e(2)) = h^0(\mathcal I_{e/\P^5}(2)) - h^0(\mathcal I_{X/\P^5}(2)) = 9 - 2 = 7$. Hence $\mathcal E_e$ is an initialized ACM bundle with $h^0 (\mathcal E_e ) = 8 = (\deg X) \cdot (\rank \mathcal E_e)$, in other words, it is an Ulrich bundle of rank 2. We refer to \cite{Beauville2016:IntroductionUlrich} for an explicit construction of such curves.

\begin{prop}[{\cite[Proposition 8]{Beauville2016:IntroductionUlrich}}]\label{Prop:UlrichRank2} There exists an Ulrich bundle of rank $2$ on $X$.
\end{prop}

\subsection{Ulrich bundles of higher ranks via Serre correspondence and \emph{Macaulay2}}

Similar as the case of cubic 3-folds in $\P^4$, the existence of rank 3 Ulrich bundles on $X$ was expected earlier in \cite[Example 4.4]{ArrondoCosta2000}. However, as Casanellas and Hartshorne pointed out \cite[Remark 5.5]{CasanellasHartshorne:Ulrich}, the construction was incorrect not only for cubic 3-folds but also for our $X$. Arrondo and Costa constructed an arithmetically Cohen-Macaulay curve $D$ of degree 15 and genus 12 using a Gorenstein liaison, however, 2 sections of $H^0 (\omega_D(-1))$ do not generate the graded module $H_{*}^0 (\omega_D)$. Indeed, in loc. cit., the authors started with a twisted cubic curve $D'$, and then found an arithmetically Gorenstein curve $B'$ of degree 18 containing $D'$ where the residual curve is $D$. Hence we have a short exact sequence
\[
0 \to \mathcal I_{B'} \to \mathcal I_{D'} \to \omega_D (-2) \to 0.
\]
Since $B'$ is arithmetically Gorenstein, we have a short exact sequence of graded $S=H_{*}^0 (\mathcal O_{\P^5})$-modules
\begin{equation}\label{Seq:aGshortexact}
0 \to H^0_* (\mathcal I_{B'}) \to H^0_* (\mathcal I_{D'}) \to H^0_* (\omega_D(-2)) \to 0.
\end{equation}
Note that $B'$ is the zero locus of a section of $\mathcal E_e(1)$, so $\mathcal I_B'$ fits into the short exact sequence $ 0 \to O_X \to \mathcal E_e (1) \to \mathcal I_{B'} (4) \to 0 $. 
Hence, the first 2 nonzero terms in the sequence (\ref{Seq:aGshortexact}) are
\[
 H^0 (\mathcal I_{D'}(1))  \simeq H^0 (\omega_D(-1))
 \] and
\[ 
H^0 (\mathcal I_{D'}(2))  \simeq H^0(\omega_D).
\]

Via the exact sequence $0 \to \mathcal I_{X/\P^5} \to \mathcal I_{D' / \P^5} \to \mathcal I_{D'} \to 0$, we may lift the sections in $H^0(\mathcal I_{D'}(j))$ as the homogeneous form of degree $j$ in $S$. It is clear that a twisted cubic curve $D' \subset \mathbb P^5$ is generated by 2 linear forms and 3 quadratic forms in $S$, and hence $H^0 (\omega_D(-1))$ is spanned by these 2 linear forms $l_1$ and $l_2$, namely.

 However,  sections in the image of $H^0(\omega_D(-1)) \otimes H^0 (\mathcal O_{\P^5} (1)) \to H^0 (\omega_D)$ can only span 11 quadrics, since two sections of $\omega_D(-1)$ admit a linear Koszul relation $l_1 l_2 - l_2 l_1 = 0$ in $S$. Hence we conclude that $H^0 (\omega_D(-1)) \otimes H^0 (\mathcal O_{\P^5}(1)) \to H^0 (\omega_D)$ cannot be  surjective.

We need to construct a curve satisfying the generating condition to construct a rank 3 Ulrich bundle $\mathcal E$ on $X$. If it exists, then two independent global sections of $\mathcal E$ will degenerate along a curve $D$ of degree $15$ since $\mathcal E$ is globally generated always. It is easy to see that the numerical conditions suggested in \cite[Example 4.4]{ArrondoCosta2000} are valid. Hence, we need to construct an ACM curve $D \subset X$  of given invariants such that $H^0(\omega_D(-1))$ has two generating section, that is, the multiplication map
\[
H^0 (\omega_D(-1)) \otimes H^0 (\mathcal O_{\P^5}(j)) \to H^0 (\omega_D(j-1))
\]
is surjective for each $j \ge 1$. 

Since $\omega_D^* (1+j)$ is nonspecial for $j \ge 2$, Castelnuovo pencil trick implies that the map is automatically surjective for $j \ge 2$. Hence it is sufficient to check only for the $j=1$ case. The construction follows from \emph{Macaulay2} \cite{Macaulay2} computations, which is analogous to \cite[Appendix]{CasanellasHartshorne:Ulrich} or \cite{Geiss:PhD}. Although the proof goes into the same strategy, in particular, the \emph{Macaulay2} scripts are almost same, it is worthwhile to write down since the difference between the cubic 3-fold case is not that much straightforward.

\begin{prop}[{See also \cite[Theorem A.3]{CasanellasHartshorne:Ulrich}}]\label{Thm:ExistenceOfACMg12d15}
The space of pairs $D \subset X \subset \P^5$ of smooth ACM curves of degree $15$ and genus $12$ on a complete intersection of $2$ quadrics $X$ has a component which dominates the Hilbert scheme of intersections of $2$ quadrics in $\P^5$. Moreover, the module $H^0_{*} (\omega_D)$ is generated by its $2$ sections in degree $-1$ as  $S_{\P^5} = H^0_{*} (\mathcal O_{\P^5})$-modules for a general pair $D \subset X$. In particular, a general intersection of two quadrics in $\P^5$ carries a desired curve we discussed above.
\end{prop}

\begin{proof}
We prove by constructing a family of such curves as in the following strategy. First, we take a family of smooth curves of genus $12$ in $\P^1 \times \P^2$. Next, we observe that a general (precisely, a randomly chosen) curve in this family admits an embedding to $\P^5$ in a natural way. Finally, we check that such a curve in $\P^5$ satisfies the desired properties. Then the whole statement will follow by the deformation theory and the semicontinuity.

Let $D$ be a smooth projective curve of genus $12$ together with line bundles $L_1$ and $L_2$ with $|L_1|$ a $\mathfrak g^1_7$ and $|L_2|$ a $\mathfrak g^2_{10}$. Let $D' $ be the image of the map
\[
D \xrightarrow{|L_1|, |L_2|} \P^1 \times \P^2.
\]
Suppose that the maps $H^0(\P^1 \times \P^2 , \mathcal O (n, m)) \to H^0 (D, L_1 ^n \otimes L_2^ m)$ are of maximal rank for all $n, m \ge 1$. Under this assumption, $D$ is isomorphic to its image $D'$ and we may compute the Hilbert series of the truncated ideal
\[
I_{trunc} = \bigoplus_{n, m \ge 3} H^0 (\mathcal I_{D'} (n,m))
\]
in the Cox ring $S_{\P^1 \times \P^2} = k[x_0, x_1; y_0, y_1, y_2]$ of $\P^1 \times \P^2$, namely,
\[
H_{I_{trunc}} (s, t) = \frac{5s^4 t^5 - 11s^4 t^4 - 6s^3 t^5 + 3s^4 t^3 + 10s^3 t^4}{(1-s)^2 (1-t)^3}.
\]
Hence, by reading off the Hilbert series, we may expect that $I_{trunc}$ admits a bigraded free resolution of type
\[
0 \to F_2 \to F_1 \to F_0 \to I_{trunc} \to 0
\]
with modules $F_0 = S_{\P^1 \times \P^2} (-3, -4)^{10} \oplus S_{\P^1 \times \P^2}(-4, -3)^{3}$, $F_1 = S_{\P^1 \times \P^2}(-3, -5)^6 \oplus S_{\P^1 \times \P^2}(-4, -4)^{11}$, and $F_2 = S_{\P^1 \times \P^2}(-4, -5)^5$.

We will construct a curve $D' \subset \P^1 \times \P^2$ in a converse direction. First, we take a free resolution of the above form, and then observe that the module represented by such a resolution is indeed an ideal of a curve $D'$. Let $M: F_2 \to F_1$ be a general map chosen randomly, and let $K$ be the cokernel of the dual map $M^* : F_1^* \to F_2^*$. The first terms of a minimal free resolution of $K$ are:
\[
\cdots \to G \stackrel{N'} \to F_1^* \stackrel{M^*} \to F_2^* \to K \to 0
\]
where $G$ be the module generated by syzygies of $M^*$. Composing $N'$ with a general map $F_0^* \to G$ and dualizing again, we get a map $N : F_1 \to F_0$. The following script shows that the kernel of $N^*$ is $S_{\P^1 \times \P^2}$ so that the entries of the matrix $S_{\P^1 \times \P^2} \to F_0^{*}$ generate an ideal.

\begin{verbatim}
i1 : setRandomSeed "RandomCurves";
     p=997;
     Fp=ZZ/p;
     S=Fp[x_0,x_1,y_0..y_2, Degrees=>{2:{1,0},3:{0,1}}]; -- Cox ring
     m=ideal basis({1,1},S); -- irrelevant ideal
\end{verbatim}
\begin{verbatim}
i2 : randomCurveGenus12Withg17=(S)->( 
     M:=random(S^{6:{-3,-5},11:{-4,-4}},S^{5:{-4,-5}}); -- random map M 
     N':=syz transpose M; -- syzygy matrix of the dual of M 
     N:=transpose(N'*random(source N',S^{3:{4,3},10:{3,4}}));
     ideal syz transpose N) -- the vanishing ideal of the curve
\end{verbatim}
\begin{verbatim}
i3 : ID'=saturate(randomCurveGenus12Withg17(S),m); -- ideal of D'
\end{verbatim}

Since the maximal rank assumption is an open condition, the above example provides that there is a component $\mathcal H \subset Hilb_{(7,10),12} (\P^1 \times \P^2)$ in the Hilbert scheme of curves of bidegree $(7,10)$ and genus $12$ defined by free resolutions of the above form. Also note that $D' \in \mathcal H$ admits both $\mathfrak g^1_{7}$ and $\mathfrak g^2_{10}$ induced by the natural projections.

We want to verify that a general $D \in \mathcal H'$ equipped with two natural projections acts like a general curve $D \in \mathcal M_{12}$, $L_1$, and $L_2$ in order to show that $\mathcal H'$ dominates $\mathcal M_{12}$. Recall from Brill-Noether theory that for a general curve $D$ of genus $g$, the Brill-Noether locus
\[
W^r_d(D) = \{ L \in \Pic (D) \ | \ \deg(L)=d, h^0(L) \ge r+1 \}
\]
is nonempty and smooth away from $W^{r+1}_d (D)$ of dimension $\rho$ if and only if
\[
\rho = \rho(g,r,d) = g-(r+1)(g-d+r) \ge 0.
\]
Also note that the tangent space at $L \in W^r_d (D) \setminus W^{r+1}_d (D)$ is the dual of the cokernel of Petri map
\[
H^0 (D, L) \otimes H^0(D, \omega_D \otimes L^{-1}) \to H^0(D, \omega_D).
\]
We expect that both $L_1$ and $L_2$ are smooth isolated points of dimension $\rho_1 = \rho_2 = 0$, equivalently, both Petri maps are injective. We refer to \cite[Chapter IV]{ACGH} for details on Brill-Noether theory.

Now let $\eta : D \to D'$ be a normalization of a given point $D' \in \mathcal H$, since we do not know that $D'$ is smooth yet. We check that $L_i$ are smooth points in the associated Brill-Noether loci as follows, where $L_i$ is a line bundle on $D$ obtained by pulling back natural $\mathfrak g^1_7$ and $\mathfrak g^2_{10}$ on $D'$ for $i=1, 2$.

We first check $L_2$; we take the plane model $\Gamma \subset \P^2$ of $D'$.

\begin{verbatim}
i4 : Sel=Fp[x_0,x_1,y_0..y_2,MonomialOrder=>Eliminate 2];
     R=Fp[y_0..y_2]; -- coordinate ring 
     IGammaD=sub(ideal selectInSubring(1,gens gb sub(ID',Sel)),R);
     -- ideal of the plane model 
\end{verbatim}
We observe that $\Gamma$ is a curve of desired degree and genus, and its singular locus $\Delta$ consists only of ordinary double points as follows.

\begin{verbatim}
i5 : distinctPoints=(J)->( 
     singJ:=minors(2,jacobian J)+J; 
     codim singJ==3) 
\end{verbatim}

\begin{verbatim}
i6 : IDelta=ideal jacobian IGammaD + IGammaD; -- singular locus 
     distinctPoints(IDelta) 
o6 = true
\end{verbatim}

\begin{verbatim}
i7 : delta=degree IDelta;
     d=degree IGammaD; 
     g=binomial(d-1,2)-delta;
     (d,g,delta)==(10,12,24) 
o7 = true
\end{verbatim}
We can also compute the minimal free resolution of $I_{\Delta}$:

\begin{verbatim}
i8 : IDelta=saturate IDelta; 
     betti res IDelta 

            0 1 2 
o8 = total: 1 4 3 
         0: 1 . . 
         1: . . . 
         2: . . . 
         3: . . . 
         4: . . . 
         5: . 4 . 
         6: . . 3 
\end{verbatim}
Thanks to the above Betti table, we immediately check that $\Gamma$ is irreducible since $\Delta$ is not a complete intersection (4,6). Indeed, there is no way to write a degree 10 curve $\Gamma \subset \P^2$ with $24$ nodes as a union of 2 curves. In particular, the normalization of $\Gamma$ is isomorphic to a smooth irreducible curve of genus $g=12$, and thus $D'$ is smooth since $12=g \le p_a(D') \le 12$. Hence from now on, we do not distinguish $D$ and $D'$ since they coincide.

By Riemann-Roch, we have $h^0(D, L_2)=3$ since $h^1(D, L_2) = h^0(D, \omega_D \otimes L_2^{-1}) = h^0(\P^2, \mathcal I_{\Delta}(6)) = 4$ by the adjunction formula applied to $D \subset \op{Bl}_\Delta \P^2$. 
Hence $|L_2|$ is complete and the Petri map for $L_2$ is identified with the muiltiplication
\[
H^0 (\P^2, \mathcal O_{\P^2}(1)) \otimes H^0 (\P^2, \mathcal I_{\Delta}(6)) \to H^0 ( \P^2, \mathcal I_{\Delta}(7)).
\]
Note that the map is injective since there is no linear relation among the 4 sextic generators of $I_{\Delta}$. In fact, the Petri map becomes an isomorphism, and $L_2 \in W^2_{10}(D)$ is a smooth isolated point of dimension $\rho_2 = 0$.

To check that $L_1$ is Petri generic, we first compute the embedding $D \to \P H^0 (\omega_D \otimes L_1^{-1}) = \P^5$ and its minimal free resolution by choosing sections of $H^0 (\omega_D) \simeq H^0 (\P^2, \mathcal I_{\Delta}(7))$ which vanish on a fiber of $D \to \P^1$ induced by $|L_1|$:

\begin{verbatim}
i9 : LK=(mingens IDelta)*random(source mingens IDelta, R^{12:{-7}});
     -- compute a basis 
     Pt=random(Fp^1,Fp^2); -- a random point in a line 
     L1=substitute(ID',Pt|vars R); -- fiber over the point 
     KD=LK*(syz(LK % gens L1))_{0..5};
     -- compute a basis for elements in LK vanish in L1
     T=Fp[z_0..z_5]; -- coordinate ring 
     phiKD=map(R,T,KD); -- embedding 
     ID=preimage_phiKD(IGammaD); 
     degree ID==15 and genus ID==12 
o9 = true
\end{verbatim}

\begin{verbatim}
i10 : betti(FD=res ID)

             0  1  2  3 4
o10 = total: 1 12 25 16 2
          0: 1  .  .  . .
          1: .  2  .  . .
          2: . 10 25 16 .
          3: .  .  .  . 2
\end{verbatim}

We observe that the curve $D \subset \P^5$ verifies the desired properties. Since the length of the minimal free resolution of $I_D$ equals to the codimension, $D \subset \P^5$ becomes ACM. Note that the dual complex $\Hom_{S_{\P^5}}^{\bullet} (F_D , S_{\P^5}(-6))$ gives a resolution of $\oplus_{n \in \Z} H^0 (\omega_D(n))$ where $F_D$ is the minimal free resolution of $D$. The Betti table also tells us that this module is generated by its 2 global sections in degree $-1$ and $h^0(L_1) = h^0(\omega_D(-1)) = 2$. Hence, $|L_1|$ is also complete and the Petri map for $L_1$ is identified with
\[
H^0(D, \omega_D(-1)) \otimes H^0(\P^5, \mathcal O_{\P^5}(1)) \to H^0(D, \omega_D).
\]
This map is also injective since there is no linear relation between the 2 generators in $H^0(\omega_D(-1))$. Indeed, the Petri map becomes an isomorphism, and $L_1 \in W^1_{7}(D)$ is a smooth isolated point of dimension $\rho_1 = 0$. As consequences, $\mathcal H$ dominates $Z = \mathcal W^1_{7} \times_{\mathcal M_{12}} \mathcal W^2_{10}$ and $\mathcal M_{12}$ thanks to Brill-Noether theory.

It remains to check the existence of a dominating family of desired curves in $\P^5$ over the space of intersections of two quadrics in $\P^5$. Since a random curve $D \in \mathcal H$ provides an embedding $D \subset \P^5$ given by a Petri generic line bundle $\mathcal O_D(1) := \omega_D \otimes L_1^{-1}$, the above construction provides a nonempty component $\mathcal H' \subset Hilb_{15t+1-12}(\P^5)$ together with a dominant rational map $\mathcal H' // Aut(\P^5) \to \mathcal M_{12}$. Note that choosing an intersection of 2 quadrics $X \subset \P^5$ containing $D$ is equivalent to choosing a 2-dimensional subspace of $H^0(\P^5, \mathcal I_{D/\P^5} (2))$. Consider the incidence variety
\[
V = \{ (D, X) \ | \ D \in \mathcal H' \text{ ACM and } X \in \op{Gr}(2, H^0(\P^5, \mathcal I_{D/\P^5}(2))) \text{ smooth} \}.
\]
Since the graded Betti numbers are upper semicontinuous in a flat family having the same Hilbert function, we observe that $V$ is birational to $\mathcal H'$ since $H^0 (\P^5, \mathcal I_{D/\P^5}(2))$ is spanned by 2 quadrics for a randomly chosen $D$. 

We compute the normal sheaf $\mathcal N_{D/X}$ for a random pair $(D, X) \in V$ as follows:

\begin{verbatim}
i11 : IX=ideal((mingens ID)*random(source mingens ID,T^{2:-2})); 
      ID2=saturate(ID^2+IX); 
      cNDX=image gens ID / image gens ID2; -- conormal sheaf 
      NDX=sheaf Hom(cNDX,T^1/ID); -- normal sheaf 
      HH^0 NDX(-1)==0 and HH^1 NDX(-1)==0 
o11 = true 
\end{verbatim}

\begin{verbatim}
i12 : HH^0 NDX==Fp^30 and HH^1 NDX==0 
o12 = true 
\end{verbatim}
In particular, the Hilbert scheme of $X$ is smooth of dimension $30$ at $[D \subset X]$, and $h^i (\mathcal N_{D/X}(-1))=0$ for $i=0, 1$. We do a similar computation for $\mathcal N_{D/{\P^5}}$:

\begin{verbatim}
i13 : cNDP=prune(image (gens ID)/ image gens saturate(ID^2)); 
      NDP=sheaf Hom(cNDP,T^1/ID); 
      HH^0 NDP==Fp^68 and HH^1 NDP==0 
o13 = true 
\end{verbatim}
Hence $\mathcal H' \subset Hilb_{15t+1-12}$ is smooth of expected dimension 68 at a general smooth point $[D \subset \P^5] \in \mathcal H'$.

Consider the natural projections 
\[
\xymatrix{
 & V \ar[dl]_{\pi_1} \ar[dr]^{\pi_2} & \\
\mathcal H' & & Gr(2, H^0(\P^5, \mathcal O_{\P^5}(2))).
}
\]
We observe that $V$ is irreducible of dimension 68 since the fiber of $\pi_1$ over $D$ is exactly a single point. Also note that the map $\pi_2$ is smooth of dimension $h^0(D, \mathcal N_{D/X}) = 30$ at $(D,X)$. Since $\dim \op{Gr}(2, H^0(\P^5, \mathcal O_{\P^5}(2))) =  38$, we conclude that $\pi_2$ is dominant. In particular, a general $X \in Gr(2, H^0(\mathcal O_{\P^5}(2)))$ contains a curve $D \in \mathcal H'$. By the semicontinuity, we conclude that a general $(D,X)$ also satisfies the desired properties.
\end{proof}
%:

Existence of such a curve $D$ on $X$ provides a construction of a rank 3 Ulrich bundle on $X$ via Serre correspondence. The idea by Casanellas and Hartshorne also makes sense in our case, and consequently, we have the following theorem:
%:
\begin{thm}[{See also \cite[Proposition 5.4 and Theorem 5.7]{CasanellasHartshorne:Ulrich}}]\label{thm:StableUlrichBundlesClassical}
Let $X \subset \P^5$ be the intersection of $2$ general quadrics in $\P^5$. Then $X$ carries an $(r^2+1)$-dimensional family of stable Ulrich bundles of rank for every $r \ge 2$.
\end{thm}

\begin{proof}
Since the strategy is almost same as in \cite{CasanellasHartshorne:Ulrich}, we only provide a shorter proof here. Note first that there is a rank 2 Ulrich bundle on any smooth complete intersection $X$ \cite{ArrondoCosta2000, Beauville2016:IntroductionUlrich}, namely, an elliptic curve type ACM bundle. Since there is no Ulrich line bundle, any rank 2 Ulrich bundle must be stable by Proposition \ref{prop: semistability of Ulrich}. Because of the same reason, if there is a rank 3 Ulrich bundle, then it is also stable.
 
Proposition \ref{Thm:ExistenceOfACMg12d15} implies that a general $X$ contains a smooth ACM curve $D$ of degree 15 and genus 12 such that $\omega_D(-1)$ has two sections which generate the graded module $H^0_{*}(\omega_D)$ as $S_{\P^5}$-modules. By Serre correspondence, those two generators define a rank 3 vector bundle $\mathcal E$ as an extension
\[
0 \to \mathcal O_{X}^{2} \to \mathcal E \to \mathcal I_{D} (3) \to 0.
\]
Since $D$ is ACM, we immediately check that $H^1 (\mathcal E(j))=0$ for every $j \in \mathbb Z$. Furthermore, we also have $H^1 (\mathcal E^{*}(j)) = H^2 (\mathcal E(-j-2)) = 0$ for every $j \in \mathbb Z$ from the dual sequence
\[
0 \to \mathcal O_X(-3) \to \mathcal E^{*} \to \mathcal O_X^{2} \to \omega_D(-1) \to 0.
\]
Hence $\mathcal E$ is an ACM bundle. Applying Riemann-Roch on $D$, we have $h^0 (\mathcal O_D(2)) = 19 = h^0 (\mathcal O_X(2))$ and thus $h^0(\mathcal E(-1)) = h^0(\mathcal I_{D}(2)) = 0$. Similarly, we have $h^0(\mathcal I_{D}(3)) = h^0(\mathcal O_X(3)) - h^0 (\mathcal O_D(3)) = 10$, and thus $h^0(\mathcal E) = 12 = (\deg X) \cdot (\rank \mathcal E)$. Indeed, $\mathcal E$ is a rank 3 Ulrich bundle. As consequences, we show the existence of Ulrich bundles on $X$ of every rank $r \ge 2$ by taking direct sums of Ulrich bundles of rank 2 and 3.

Suppose first that we have a stable Ulrich bundle $\mathcal E$ of rank $r$ for every $r \ge 2$. By Riemann-Roch, we have $\chi(\mathcal E \otimes \mathcal E^{*}) = -r^2$. Since the computations in \cite[Proposition 5.6]{CasanellasHartshorne:Ulrich} also holds for our $X$,  we have $h^2( \mathcal E \otimes \mathcal E^{*}) = h^3 ( \mathcal E \otimes \mathcal E^{*}) = 0$. Since $\mathcal E$ is simple, we conclude that $h^0 ( \mathcal E \otimes \mathcal E^{*}) = 1$ and $h^1( \mathcal{E} \otimes \mathcal E^{*}) = r^2 + 1$ as desired. Hence the moduli space of stable Ulrich bundles is smooth of expected dimension if it is nonempty.

It only remains to show the existence of stable Ulrich bundles of rank bigger than $3$. Let $r \ge 4$, $\mathcal E'$ and $\mathcal E'' \not \simeq \mathcal E'$ be stable Ulrich bundles of rank $2$ and $r-2$, respectively. By Riemann-Roch and \cite[Proposition 5.6]{CasanellasHartshorne:Ulrich}, we have $h^1( \mathcal E' \otimes \mathcal {E''} ^*) = -\chi( \mathcal E' \otimes \mathcal {E''} ^*) = 2r - 4 > 0$. Hence the space $\P \Ext_X ^1 (\mathcal E'' , \mathcal E')$ is nonempty and each element gives a nonsplit extension
\[
0 \to \mathcal E' \to \mathcal E \to \mathcal E'' \to 0
\]
where $\mathcal E$ is a simple and strictly semistable Ulrich bundle of rank $r$. Such extensions form a family of dimension
\[
\dim \{ \mathcal E' \} + \dim \{ \mathcal E'' \} + \dim \P \Ext_X^1 (\mathcal E'' , \mathcal E') = r^2-2r+5 < r^2+1.
\]
Since all the other extensions by different ranks form smaller families, we conclude that a general  Ulrich bundle of rank $r$ is stable. This completes the proof.
\end{proof}
%:

%:
\begin{rmk} We finish this section by a few remarks.
\begin{enumerate}
\item In fact, the proof of Proposition \ref{Thm:ExistenceOfACMg12d15} implies much stronger results. For instance, one can check that $\mathcal H$ is a unirational family which dominates the moduli space $\mathcal M_{12}$ of smooth curves of genus $12$ as in \cite[Appendix]{CasanellasHartshorne:Ulrich}.
\item Because we made a computer-based computation over a finite field, we cannot remove the assumption $X$ being general. It is also mysterious that ``how general'' $X$ should be. 
\item As we mentioned, the above approach closely follows \cite{CasanellasHartshorne:Ulrich}. In loc. cit., the authors also checked that any smooth cubic 3-fold contains an elliptic normal curve of degree 5. Similarly, any smooth complete intersection of two quadrics in $\P^5$ contains an elliptic normal curve of degree 6, as in \cite[Proposition 8]{Beauville2016:IntroductionUlrich}. It is an interesting task to construct smooth ACM curves of degree 15 and genus 12 on any smooth complete intersection of two 4-dimensional quadrics. 
\end{enumerate}
\end{rmk}

%:
\section{Derived categorical approaches}\label{section:derived}
The notion of semiorthogonal decomposition enables us to reduce problems about Ulrich bundles on $X$ to problems about vector bundles on the associated curve $C$. Let us recall some necessary facts about the moduli space of vector bundles on curves and the derived category of coherent sheaves on $X$.

\subsection{Stable vector bundles on curves}

	Let $C$ be a smooth projective curve of genus $g$, $\mathcal U_C(r,d)$ be the moduli space of S-equivalence classes of rank $r$ semistable vector bundles of degree $d$ on $C$, and $\mathcal{SU}_C(r, L)$ be the moduli space of S-equivalence classes of rank $r$ semistable vector bundles of determinant $L$ on $C$. We use the superscript $(\text{--})^{\sf s}$ to describe the sub-moduli space parametrizing stable objects. It is well-known that $\mathcal U_C(r,d)$ and $\mathcal{SU}_C(r, L)$ are normal projective varieties (see \cite{NarasimhanRamanan:ModuliofVectBdl,Seshadri:SpaceofUnitaryVectBdls}).
	
	The lemma below is one of the well-known results for (semi-)stable bundles on curves.
	\begin{lem}
		Let $ F$ be a stable vector bundle of rank${}\geq 2$ on $C$. Then,
		\begin{enumerate}
			\item $\mu( F) \geq 2g-2$ implies $h^1( F)=0$, and
			\item $\mu( F) \geq 2g-1$ implies that $ F$ is globally generated.
		\end{enumerate}
		If the inequalities on $\mu$ are strict, then the same results are valid for $F$ semistable.
	\end{lem}
	\begin{proof}
		Assume that $h^1( F) \neq 0$. Then $h^0( F^* \otimes \omega_C)\neq 0$, which is imposible unless $ F = \omega_C$ since $ F$ is stable and $\deg( F^*\otimes \omega_C) \le 0$. This proves (1). If $\mu( F) \geq 2g-1$, then $H^1( F(-P))=0$ for any $P \in C$, hence $H^0( F) \to  F \otimes \kappa(P)$ is surjective. Using Nakayama's lemma, we conclude that $H^0( F) \otimes \mathcal O_C \to  F$ is surjective.
	\end{proof}
%:
	\begin{lem}[{\cite[Exercise~2.8]{Popa:GeneralizedTheta}}]\label{lem: semistability in Popa note}
		Let $ F,\, G$ be vector bundles on $C$ such that $H^p(  F\otimes  G)=0$ for $p=0,1$. Then both $ F$ and $ G$ are semistable.
	\end{lem}
	\begin{proof}
		By Riemann-Roch, $\mu(  F \otimes  G)= g-1$. Assume that there exists $0 \neq  F' \subset  F$ such that $\op{rank} F' < \op{rank} F$ and $\mu( F') > \mu( F)$. Then, $\mu( F' \otimes G) > \mu( F \otimes  G) = g-1$. This shows that $\chi( F' \otimes  G) > 0$, in particular, $h^0( F' \otimes  G) > 0$. This contradicts to $ F' \otimes G \subset  F \otimes  G$ and $h^0( F \otimes  G)=0$. It follows that $F$ is semistable, and the same argument applies to $ G$.
	\end{proof}
%:

	Similar as in the case of line bundles, we may define the Brill-Noether locus as follows:
\[
W^{k-1}_{r,d}(C) := \{  [F] \in \mathcal U_C^{\sf s} (r,d) \ | \ h^0(C,  F) \ge k \}
\]
which is a subscheme of $\mathcal U_C^{\sf s} (r,d)$ of expected dimension $\rho^{k-1}_{r,d} = r^2(g-1) + 1 - k(k-d+r(g-1))$. The following theorem is useful in the future:

\begin{thm}[{\cite[Theorem B]{BGN1997}}] \label{Theorem:BGN}
The locus $ W^{k-1}_{r,d} (C)$ is nonempty if and only if 
\[
d>0, r \le d+(r-k)g \text{ and } (r,d,k) \neq (r,r,r).
\]
\end{thm}

\subsection{Derived categories of $X$}

Let $Q^n \subset \P^{n+1}$ be a smooth quadric hypersurface. We unify all the notations which involve spinor bundles in accordance with \cite{BondalOrlov:SODforAlgVar}. Hence, the spinor bundles on the quadric $Q^n$ give the semiorthogonal decomposition\,\cite{Kapranov:DerivedCat_Homogeneous}
\[
	\D(Q^n) = \left\{
		\begin{array}{ll}
			\bigl\langle\, \mathcal O(-n+1),\,\ldots ,\,\mathcal O,\,S \,\bigr\rangle & \text{if $n$ is odd} \\
			\bigl\langle\, \mathcal O(-n+1),\,\ldots ,\,\mathcal O,\,S^+,\,S^- \,\bigr\rangle & \text{if $n$ is even}
		\end{array}
	\right.
\]
	Especially in the case $n=4$, $S^\pm$ correspond to the universal quotient bundle and the dual of the universal subbundle under the isomorphism $Q^4 \simeq \op{Gr}(2, \C^4)$. \par
%:
	Let $Q_0,\, Q_\infty \subset \P^5$ be two nonsingular $4$-dimensional quadrics whose intersection defines $X$. Without loss of generalities, we may assume
	\[
		Q_0 = ( x_0^2 + \ldots + x_5^2 = 0 )\quad\text{and}\quad Q_\infty = ( \lambda_0 x_0^2 + \ldots + \lambda_5 x_5^2 = 0 )
	\]
	for some $\lambda_0,\, \ldots, \lambda_5 \in \C$. We define $X := Q_0 \cap Q_\infty$ a smooth threefold of degree 4. One well known approach to $X$ is to associate the quadric pencil $\mathfrak d := \lvert Q_0 + t Q_\infty \rvert_{t \in \P^1}$ on $\P^5$. Let us assume that the pencil $\mathfrak d$ is nonsingular in the sense of \cite{Reid:PhD}, namely,  each singular quadric $Q_{\lambda_i}$\,($i=0,\ldots,5$) is isomorphic to the cone of a smooth quadric $Q^3 \subset \P^4$ over a point. Note that this condition is equivalent to saying that $\lambda_0,\ldots,\lambda_{5}$ are pairwise distinct. Also note that none of $\lambda_0,\ldots,\lambda_{5}$ is zero since $Q_{\infty}$ is smooth.

	The resolution of indeterminacy of $\varphi_\mathfrak d \colon \P^5 \dashrightarrow \P^1$ gives the relative quadric $\mathcal Q \to \P^1$. Let $\sigma \colon C \to \P^1$ be the double cover ramified over $[1:\lambda_0],\,\ldots,\,[1:\lambda_5] \in \P^1$, and let $\mathcal Q_C := \mathcal Q \times_{\P^1} C$ be the fiber product. Bondal and Orlov\,\cite{BondalOrlov:SODforAlgVar} showed that $C$ is the fine moduli space of spinor bundles on the quadrics in $\mathfrak d$, {\it i.e.} there exists a vector bundle $\mathcal S_{\mathcal Q_C}$ on $\mathcal Q_C$ such that for each $c \in C$, the restriction $\mathcal S_{\mathcal Q_C}\big\vert_{\mathcal Q \times\{c\}}$ is one of the spinor bundles on the quadric $Q_{\sigma(c)}$. When $Q_{\sigma(c)}$ is a singular quadric, then it is a cone $\mathcal C(  Q^3 )$ of a $3$-dimensional quadric over a point $v \in \mathbb P^5$. In this case $\mathcal S_{\sigma(c)}$ is the pullback of the unique spinor bundle on $Q^3$ by $\mathcal C(  Q^3) \setminus \{v\} \to Q^3$. We define the vector bundle $\mathcal S := \mathcal S_{\mathcal Q_C}\big \vert_{X \times C}$.
	 
\begin{thm}[Bondal--Orlov\,\cite{BondalOrlov:SODforAlgVar}]\label{thm: Bondal-Orlov}
		The Fourier--Mukai transform 
		\[
			\Phi_{\mathcal S} \colon \D(C) \to \D(X),\  F^\bullet \mapsto Rp_{X*}( Lp_C^* F^\bullet \Dotimes \mathcal S)
		\]
		is fully faithful, and induces a semiorthogonal decomposition
		\[
			\D(X) = \bigl\langle\, \mathcal O_X(-1),\,\mathcal O_X,\,\Phi_{\mathcal S}\bigl(\D(C)\bigr)\,\bigr\rangle.
		\]
	\end{thm}
%:
	Furthermore, $X$ can be regarded as the fine moduli space of stable vector bundles of rank $2$ with fixed determinant of odd degree\,\cite{Newstead:StableBundlesofRank2OddDeg}, and $\mathcal S$ is the universal bundle of this moduli problem. There arises an ambiguity of the choice of this fixed determinant\,(the theorem of Bondal and Orlov is independent of the replacement $\mathcal S \mapsto \mathcal S \otimes p_C^*  L$ for any line bundle $ L \in \Pic C$).
%:
\begin{defi}\label{def: univ Spinor}
	We choose $\xi$ a line bundle of degree $1$, and assume that $\mathcal S$ is the universal family of the fine moduli space $\mathcal{SU}_C(2,\xi^*) \simeq X$ which parametrizes the stable vector bundles of rank $2$ and determinant $\xi^*$. Equivalently, $\mathcal S$ is determined by imposing the condition $\det \mathcal S = \mathcal O_X(1) \boxtimes \xi^*$.
\end{defi}	 
	This choice of $\mathcal S$ is precisely dual to the same symbol in Section 5 of \cite{Kuznetsov:Instanton}. We remark that some parts of the next subsection are following the arguments in \cite{Kuznetsov:Instanton}. This may cause confusions, so we rephrase the details which are necessary for the rest part of the paper.
%:

	\subsection{Ulrich bundles via derived categories}
%:
	Let $\op{Coh}(X)$ be the category of coherent sheaves on $X$. There is a natural functor $\op{Coh}(X) \to \D(X)$ which maps a coherent sheaf $\mathcal E$ to the complex concentrated at degree zero:
	\[
		\ldots \to 0 \to \mathcal E \to 0 \to \ldots.
	\]
	This identifies $\op{Coh}(X)$ to a full (but not triangulated) subcategory of $\D(X)$, hence we may regard a coherent sheaf on $X$ as an object in $\D(X)$. Conversely, we call an object $\mathcal E^\bullet \in \D(X)$ a coherent sheaf\,(resp. a vector bundle) if $\mathcal E^\bullet$ is isomorphic to an object\,(resp. a locally free sheaf) in $\D(X) \cap \op{Coh}(X)$.
	
	We use derived categories to classify Ulrich bundles on $X$. We first assume that there exists an Ulrich bundle $\mathcal E$ of rank $r \geq 2$ on $X$\,(the existence will be proved later). By Proposition~\ref{prop: Ulrich Equiv conditions}, $H^p(\mathcal E(-i)) = \Hom_{\D(X)}( \mathcal E^*(1), \mathcal O(-i+1)[p])=0$ for all $p$ and $i=1,2,3$.
	Using the semiorthogonal decomposition in Theorem~\ref{thm: Bondal-Orlov}, one immediately sees that $\mathcal E^*(1) \in \Phi_{\mathcal S} \D(C)$. Since $\D(C) \to \Phi_\mathcal S(\D(C))$ is an equivalence of categories, the study of Ulrich bundles on $X$ boils down to the study of certain objects in $\D(C)$. Such objects are obtained by mapping $\mathcal E^*(1)$ along the projection functor $\Phi_\mathcal S^! \colon \D(X) \to \D(C)$. Before to proceed, let us note that the projection $\Phi_\mathcal S^!$ is right adjoint to $\Phi_\mathcal S$. Since the functor $\Phi_\mathcal S$ is given by $ F \mapsto Rp_{X*}(Lp_C^* F \Dotimes \mathcal S)$ where $p_X \colon X \times C \to X$ and $p_C \colon X \times C \to C$ are the natural projections, its right adjoint has the following form\,(\textit{cf.} \cite[Proposition~5.9]{Huybrechts:FourierMukai}):
	\[
		\Phi^!_\mathcal S \colon \D(X) \to \D(C),\qquad \mathcal E \mapsto Rp_{C*}\bigl( Lp_X^* \mathcal E \Dotimes \mathcal S^*) \otimes \omega_C [1].
	\]
%:
%:
	Meanwhile, the Ulrich conditions in Proposition~\ref{prop: Ulrich Equiv conditions}-\ref{item: Ulrich as Acyclic condition} impose an extra condition on $\mathcal E^*(1)$ other than $\mathcal E^*(1) \in \Phi_\mathcal S \D(C)$. Indeed, the condition $H^\bullet( \mathcal E(-3))=0$ is not followed by $\mathcal E^*(1) \in \Phi_\mathcal S \D(C)$. It can be expressed as follows:
	\begin{align}
		& \Hom_{\D(X)}(\mathcal E^*(1),\, \mathcal O_X(-2)[p]) = 0 \nonumber \\
		\Leftrightarrow& \Hom_{\D(X)}(\Phi_\mathcal S \Phi_\mathcal S^!( \mathcal E^*(1)),\, \mathcal O_X(-2)[p])=0  \nonumber \\
		\Leftrightarrow& \Hom_{\D(C)}(\Phi_\mathcal S^!(\mathcal E^*(1)),\, \Phi_\mathcal S^! (\mathcal O_X(-2))[p])=0. \label{eq: Ulrich orthogonal condition}
	\end{align}
%:
%:
	\begin{lem}\label{lem: Projection image and Raynaud bundle}
		We have $\Phi_{\mathcal S}^! (\mathcal O_X(-2))[2] \simeq \mathcal R^* \otimes \omega_C^{\otimes 2}$, where $\mathcal R$ is the second Raynaud bundle which appears in \cite[Section 5.4]{Kuznetsov:Instanton}.
	\end{lem}
	\begin{proof}
		By \cite{Kuznetsov:Instanton}, $\mathcal R \simeq \Phi_{\mathcal S^*}^! \mathcal O_X[-1] = p_{C*} \bigl( \mathcal S \otimes p_X^* \mathcal O_X) \otimes \omega_C$. Thus,
		\begin{align*}
			\mathcal R^* &\simeq \varHom_{\D(C)}( p_{C*} \mathcal S \otimes \omega_C,\ \mathcal O_C ) \\
			&\simeq p_{C*} \varHom_{\D(X \times C)} ( \mathcal S \otimes p_C^* \omega_C,\ p_X^*\omega_X [3] )\\
			&\simeq p_{C*} \bigl( \mathcal S^* \otimes p_X^*\mathcal O_X(-2) \bigr) \otimes \omega_C^* [3] \\
			&\simeq \Phi_{\mathcal S}^!( \mathcal O_X(-2)) \otimes \omega_C^{\otimes(-2)}[2],
		\end{align*}
		where the second isomorphism is given by Grothendieck-Verdier duality.
	\end{proof}
%:
	Together with the orthogonality condition (\ref{eq: Ulrich orthogonal condition}), we have to understand how the object $\Phi_\mathcal S^!( \mathcal E^*(1))$ looks like. One standard way is to analyze the restriction to the point $\Phi_\mathcal S^! \mathcal E^*(1) \otimes \kappa(c) \in \D(\{c\})$. We fix the notations to avoid confusion as follows.
	\begin{nota}
		For $x \in X$, we denote by $\mathcal S_x$ the vector bundle over $C$ determined by the restriction of $\mathcal S$ to $\{x\} \times C \simeq C$. Similarly, the vector bundle $\mathcal S_c$\,($c \in C$) over $X$ is defined to be the restriction of $\mathcal S$ to $X \times\{c\} \simeq X$.
	\end{nota}
%:
	The proof of the following proposition is essentially due to \cite[Theorem~5.10]{Kuznetsov:Instanton}, but we write down the proof to prevent the confusions arising from the choice of a convention.
%:
	\begin{prop}\label{prop:UlrichViaDerivedCategory}
		Suppose there exists an Ulrich bundle $\mathcal E$ of rank $r$ on $X$. Then, $ F := \Phi_\mathcal S^!(\mathcal E^*(1)) \in \D(C)$ is a semistable vector bundle over $C$ of rank $r$ and degree $2r$. Furthermore, $ F$ satisfies
		\begin{enumerate}
			\item $\Ext^p_C( \mathcal R,\,  F^* \otimes \omega_C^{\otimes 2} ) = 0$ for $p=0,1$ and
			\item $H^1( F \otimes \mathcal S_x)=0$ for each $x \in X$.
		\end{enumerate}
		Conversely, if $ F$ is a semistable vector bundle over $C$ of rank $r$ and degree $2r$ satisfying the conditions (1) and (2) above, then $\Phi_\mathcal S  F = \mathcal E^*(1)$ for some Ulrich bundle $\mathcal E$ over $X$.
	\end{prop}
	\begin{proof}
		Let $c \in C$ be a point. Then $ F \otimes \kappa(c) \in \D(\{c\})$ is the complex of $\C$-vector spaces whose cohomology sheaves are controlled by
		\begin{equation}\label{eq: FM projection of E^*(1) at a point}
			H^{p + 1} ( X,\ \mathcal E^*(1) \otimes \mathcal S_c^*) \simeq \Ext^{p+1}_X(\mathcal E(-1),\, \mathcal S_c^*).
		\end{equation}
		By \cite[p.~310]{Ottaviani:Spinor}, $\mu(\mathcal S_c^*) = -1/2$, regardless whether $c$ is a ramification point or not. Hence $\mu(\mathcal E(-1))=0$, Proposition~\ref{prop: Stability vanishing}, and Serre duality imply that
		\begin{align*}
			\Ext_X^0(\mathcal E(-1),\,\mathcal S_c^*) &\simeq \Hom_X(\mathcal E(-1),\, \mathcal S_c^*) = 0.
		\end{align*}
		Consider the following short exact sequence (\textit{cf.} \cite[Theorem~2.8]{Ottaviani:Spinor})
		\begin{equation}\label{eq: Ottaviani Seq on X}
			0 \to \mathcal S_{\tau c}^* \to \mathcal O_X^{\oplus4} \to \mathcal S_{c}^*(1) \to 0
		\end{equation}
		where $\tau \colon C \to C$ is the hyperelliptic involution arising from the double cover $C \to \P^1$. Note that even for the ramification points $c \in C$, one can compose the sequence (\ref{eq: Ottaviani Seq on X}) in a natural way. Tensoring (\ref{eq: Ottaviani Seq on X}) with $\mathcal E^*(j)$ for $j=-1,0,1$, we have $H^{p+1}(\mathcal E^*(1) \otimes \mathcal S_c^*) \simeq H^{p+2}( \mathcal E^* \otimes \mathcal S_{\tau c}^* ) \simeq H^{p+3}(\mathcal E^*(-1) \otimes \mathcal S_c^*)$ and the latter one vanishes for $p \geq 1$. This proves that (\ref{eq: FM projection of E^*(1) at a point}) is zero unless $p = 0$, in other words, $ F$ is a coherent sheaf concentrated at degree $0$. Furthermore, since $p_X^*(\mathcal E^*(1)) \otimes \mathcal S$ is flat over $C$, $c \mapsto \chi( \mathcal E^*(1) \otimes \mathcal S^*_c)$ is a constant function and thus $F$ is a vector bundle on $C$.
		\par
%:
		To compute $\op{rank} F$ and $\deg  F$, we use Grothendieck-Riemann-Roch which reads
		\begin{align}
			\op{ch} (\Phi_\mathcal S F) &= \op{ch} ( Rp_{X*}( p_C^*  F \otimes \mathcal S) ) = p_{X*}\bigl( \op{ch}( p_C^* F) \op{ch}(\mathcal S) \op{td}(\mathcal T_{p_X})\bigr) \nonumber\\
			&= (2d - 3s) + \frac13 (2s-d) P_X - sL_X + (d-2s)H_X, \label{eq: G-R-R for FM Transform}
		\end{align}
		where $d = \deg  F$ and $s = \op{rank} F$. The computation method is identical to the one introduced in \cite[Lemma~5.2]{Kuznetsov:Instanton} except that the Fourier-Mukai kernels are dual to each other.
		
		Since $\Phi_\mathcal S  F = \mathcal E^*(1)$ is of rank $r$ and of degree zero, we find $2d-3s = r$ and $d-2s = 0$. It follows that $s = r$ and $d = 2r$. By (\ref{eq: Ulrich orthogonal condition}) and Lemma~\ref{lem: Projection image and Raynaud bundle}, 
		\begin{align*}
			\Hom_{\D(C)}(\Phi_S^!(\mathcal E^*(1)),\, \Phi_\mathcal S^! (\mathcal O_X(-2))[p]) 
			&\simeq \Hom_{\D(C)}( F ,\, \mathcal R^* \otimes \omega_C^{\otimes 2}[p-2]) \\
			&\simeq \Ext^{p-2}_C( \mathcal R ,\,  F^* \otimes \omega_C^{\otimes 2}).
		\end{align*}
		Since both $\mathcal R$ and $ F$ are vector bundles, it suffices to require $\Ext^p_C(\mathcal R,\,  F^* \otimes \omega_C^{\otimes 2})=0$ for $p=0,1$ to fulfill (\ref{eq: Ulrich orthogonal condition}). The semistability of $ F$ follows from Lemma~\ref{lem: semistability in Popa note}. Finally, $H^1( F \otimes \mathcal S_x)=0$ follows from the fact that $\Phi_\mathcal S  F = \mathcal E^*(1)$ is a vector bundle on $X$; indeed, $\Phi_\mathcal S  F = Rp_{X*}( p_C^*  F \otimes \mathcal S)$ is the complex concentrated at zero, hence $R^1p_{X*}( p_C^* F \otimes \mathcal S)=0$. By the cohomology base change, $H^1( F \otimes \mathcal S_x)=0$ for each $x \in X$.
		
		Conversely, assume that $ F$ is a semistable vector bundle on $C$ satisfying all the prescribed conditions. The condition (2) implies that $\Phi_\mathcal S  F \in \D(X)$ is a vector bundle on $X$. Then $\Phi_\mathcal S  F \in \Phi_\mathcal S \D(C)$ together with (1) can be interpreted as $\Ext_X^p( \Phi_\mathcal S  F,\, \mathcal O_X(-j))=0$ for $j=0,1,2$, showing that $\mathcal E:= (\Phi_\mathcal S  F)^* \otimes \mathcal O_X(1)$ is an Ulrich bundle over $X$.
	\end{proof}
%:
	Using (\ref{eq: G-R-R for FM Transform}) and $\Phi_\mathcal S  F = \mathcal E^*(1)$, we can immediately check that
	\[
		(\,c_i(\mathcal E)\,)_i = (\ 1,\, r,\, 2r^2-r,\, \tfrac{1}{3}r(r-2)(2r+1)\ ).
	\]
%:
%:
%:
	Proposition~\ref{prop:UlrichViaDerivedCategory} gives a bijection between the set of Ulrich bundles on $X$ with the set of certain semistable vector bundles on $C$. From now on, we bring our focus into the semistable vector bundles on $C$ satisfying the conditions described in Proposition~\ref{prop:UlrichViaDerivedCategory}. First of all, we prove that a general stable bundle in $\mathcal U_C(r,2r)\,(r \geq 2)$ satisfies the condition (2) of Proposition~\ref{prop:UlrichViaDerivedCategory}.

	\begin{prop}\label{prop: generic FM image is locally free}
		For $r \geq 2$, let $\mathcal U_C^{\sf s}(r,2r)$ be the moduli space of stable vector bundles on $C$ of rank $r$ and degree $2r$. The subset 
		\[
			\bigl\{ [ F] \in \mathcal U_C^{\sf s}(r,2r) : h^1( F \otimes \mathcal S_x)=0\ \text{for every }x \in X\bigr\}
		\]
		is open and nonempty.
	\end{prop}	
	\begin{proof}
		First of all, we claim that the set $\{[ F] \in \mathcal U_C^{\sf s}(r,2r) : h^1(F \otimes \mathcal S_x)=0\ \text{for every }x \in X\}$ is open in $\mathcal U_C^{\sf s}(r,2r)$. Consider the closed subset $Z \subset X \times \mathcal U_C^{\sf s}(r,2r)$ defined by
		\[
			\{ (x,[ F]) : h^1( F \otimes \mathcal S_x) \geq 1 \}.
		\]
		Since the projection morphism $\op{pr}_2 \colon X \times \mathcal U_C^{\sf s}(r,2r) \to \mathcal U_C^{\sf s}(r,2r)$ is proper, $V := \mathcal U_C^{\sf s}(r,2r) \setminus \op{pr}_2(Z)$ is open in $\mathcal U_C^{\sf s}(r,2r)$. Writing down the locus $V$ set-theoretically, we can easily find that
		\[
			V = \{ [ F] \in \mathcal U_C^{\sf s}(r,2r) : h^1( F \otimes \mathcal S_x)=0\ \text{for every }x \in X\}. 
		\]

		For $r=2$, we know that any smooth $X$ carries an Ulrich bundle $\mathcal E$ of rank 2 as in Proposition~\ref{Prop:UlrichRank2}. Note that its projection image $ F := \Phi_{\mathcal S}^{!} (\mathcal E^* (1))$ is a rank 2 vector bundle of degree $4$ on $C$ satisfying the desired property. Assume that $r \ge 3$. Let $ F$ be a stable vector bundle of rank $r$ and degree $2r$, and let $x \in X$. Suppose that $H^1( F \otimes \mathcal S_x) \simeq \Hom_C ( F, \mathcal S_x ^* \otimes \omega_C)^*$ is nonzero. By the stability condition, any nonzero morphism $ F \to \mathcal S_x ^* \otimes \omega_C$ must be surjective, so we have a short exact seqeunce
		\[
			0 \to  F' \to  F \to \mathcal S_x^* \otimes \omega_C \to 0
		\]
		where $ F'$ is a semistable vector bundle of rank $(r-2)$ and degree $(2r-5)$. By Riemann-Roch, we have $ext^1_C (\mathcal S_x^* \otimes \omega_C,  F') = 3r-4$. Hence, for each $x \in X$, the locus of vector bundles $ F $ fit into the above exact sequence has dimension at most $(r-2)^2 + 1 + (3r-5) = r^2 - r$. As varying $x \in X$, the bad locus can sweep out a set of dimension at most $r^2 - r + 3 < r^2 + 1$. Hence we conclude that a general $ F \in \mathcal U_C^{\sf s}(r, 2r)$ does not admit a surjection to $S_x^* \otimes \omega_C$ for any $x \in X$.		
	\end{proof}
%:
\begin{rmk}
The formula (\ref{eq: G-R-R for FM Transform}) tells us that there is no line bundle $ F$ of degree 2 such that $\Phi_{\mathcal S} F$ is locally free. Indeed, there is no line bundle $\mathcal E$ on $X$ such that $\op{ch} (\mathcal E) = 1 - L_X$. In particular, there is no Ulrich line bundle on $X$.
\end{rmk}
%:

%:
%:
	Our aim is to find a semistable vector bundle $ F$ of rank $r$ and degree $2r$ such that $\Ext^p_C( \mathcal R,\,  F^* \otimes \omega_C^{\otimes 2})=0$ for $p=0,1$. Since $ G :=  F^* \otimes \omega_C^{\otimes 2}$ is also a semistable vector bundle of rank $r$ and degree $2r$, the following proposition guarantees the existence of Ulrich bundles at least when $r=3$:
	\begin{prop}\label{prop: Generic orthogonality r=3}
		$\Hom_C(\mathcal R,  G) = 0$ for a generic stable vector bundle $G$ of rank $3$ and degree $6$.
	\end{prop}
	\begin{proof}
		Suppose that there is a nontrivial morphism $\mathcal R \to G.$ Note that $\mathcal R$ is a stable vector bundle\,\cite[Corollary~6.2]{Hein:Raynaud}. By the stability condition, we observe that the image of $\mathcal R \to G$ is either a rank 2 vector bundle of degree 3, or a rank 3 vector bundle of degree 4, 5, 6. We show by cases that these conditions are not generic. 

\begin{enumerate}
\item Suppose that the image of $\mathcal R \to  G$ is a rank 2 vector bundle of degree 3. There are two short exact sequences 
\[0 \to  G'' \to \mathcal R \to  G' \to 0 \]
  and
  \[ 0 \to  G' \to  G \to  L \to 0 \]
   where $ G'$ is the image of $\mathcal R$, $ G''$ is a rank 2 vector bundle of degree 1. Note that both $ G'$ and $ G''$ are stable. Also, $ L$ is locally free: indeed, if $\bar{ G}'$ is the kernel of the morphism $ G \to  L / \op{Tors} L$, then the stability argument forces that $ G' = \bar{ G}'$, hence $ L =  L / \op{Tors} L$ showing that $ L$ is locally free. Since $h^0(C,  G')>0$, a nonzero section $s \in H^0(C,  G')$ defines the following exact sequence
   \[
   	0 \to \mathcal O_C(D) \stackrel{s} \to  G' \to  M \to 0 
   \]
where $D$ is the zero locus $V(s)$ of $s$ and $ M = \det  G' \otimes \mathcal O_C(-D)$ is a line bundle. By the stability, we have either $\deg D = 0$ or $1$. Tensoring by $ G''^*$, we have
\[
	0 \to  G''^* (D) \to  G' \otimes  G''^* \to  G''^* \otimes  M \to 0.
\]

When $\deg D=0$, that is, $D=0$, the stability of $ G''$ assures that
\begin{align*}
	\dim \Hom_C ( G'',  G') & \le h^0(C,  G''^* ) + h^0 (C,\,  G''^*  \otimes  M) \\
	& = 0 + 3 = 3.
\end{align*}
When $\deg D=1$, 
\begin{align*}
	\dim \Hom_C ( G'',  G') & \le h^0(C,  G''^* (D) ) + h^0 (C,\,  G''^* \otimes  M) \\
	& =  1 + 2 = 3
\end{align*}
since both the Brill-Noether loci $W^{1}_{2,1} (C)$ and $W^{2}_{2,3}(C)$ are empty by Theorem \ref{Theorem:BGN}. In any cases, we observe that the Quot scheme $[\mathcal R \to  G'] \in \op{\it Quot}_{2,3} (\mathcal R)$ has the local dimension at most 3 for any stable quotient $ G' \in \mathcal U_C^{\sf s}(2,3)$. The locus of vector bundles $ G \in \mathcal U_C^{\sf s}(3, 6)$ which is an extension of $ L$ by $ G'$ has the dimension at most
\begin{align*}
\dim \{  G \} & \le \dim  \op{\it Quot}_{2,3} (\mathcal R) + \dim  \Pic^3 (C) + \dim \P\Ext^1_C ( L,  G) \\
& \le 3 + 2 + 4  = 9\\
& < 10 = \dim  \mathcal U_C^{\sf s} (3,6).
\end{align*}

\item Suppose that the image of $\mathcal R$ is a rank 3 vector bundle of degree 4. We have  two short exact sequences $$ 0 \to  L \to \mathcal R \to  G' \to 0 $$ and $$ 0 \to  G' \to  G \to  T \to 0 $$ where $ L$ is a line bundle of degree 0, $ G'$ is the image of $\mathcal R$, and $ T$ is a torsion sheaf of length 2. Since $\dim \Hom_C( L,\mathcal R) = 1$\,({\textit{cf.}} the proof of \cite[Lemma~5.9]{Kuznetsov:Instanton}), the dimension of the family of stable vector bundles $ G' \in \mathcal U_C(3,4)$ which fit into the first short exact sequence is at most $\dim \Pic^0 (C) = 2$. Hence the dimension of the family of stable vector bundles $ G$ which fit into the second short exact sequence is at most $\dim \{  T \} + \dim \{  G' \} + \dim \P \Ext^1_C ( T,  G') = 9$.

\item Suppose that the image of $\mathcal R$ is a rank 3 vector bundle of degree 5. We have  two short exact sequences $$ 0 \to  L \to  \mathcal R \to  G' \to 0 $$ and $$ 0 \to  G' \to  G \to  T \to 0 $$ where $ L$ is a line bundle of degree $-1$, $ G'$ is the image of $\mathcal R$, and $ T$ is a torsion sheaf of length 1. Since $\mathcal R$ is stable, we have $\dim \Ext^1_C ( L,\mathcal R)= \dim \Hom_C(\mathcal R,\, L \otimes \omega_C)=0$. By Riemann-Roch, we have $\dim \Hom_C( L, \mathcal R)=4$, and thus the dimension of the family of stable vector bundles $ G' \in \mathcal U_C(3,5)$ which fit into the first exact sequence is at most $\dim \Pic^{-1}(C) + \dim \P \Hom_C ( L, \mathcal R) = 5$. Therefore, the dimension of the family of stable vector bundles $ G$ which fit into the second exact sequence is at most $\dim \{  T \} + \dim \{  G' \} + \dim \P \Ext^1_C(  T,  G') = 8$.

\item Suppose that the image of $\mathcal R$ is a rank 3 vector bundle of degree 6, in other words, it coincides with $ G$. We have the following short exact sequence $$ 0 \to  L \to \mathcal R \to  G \to 0 $$ where $ L$ is a line bundle of degree $-2$. By the stability and Riemann-Roch formula, we have $\dim \Hom_C( L, \mathcal R)=\chi( L,\mathcal R)=8$. Hence the dimension of the family of stable vector bundles $G$ which fits into the above exact sequence is at most $\dim \Pic^{-2}(C) + \dim \P \Hom_C( L, \mathcal R) = 9$.

\end{enumerate}
To sum up, we conclude that a generic stable vector bundle $ G \in \mathcal U_C(3,6)$ yields $\Hom_C(\mathcal R,  G) = 0$.\qedhere
\end{proof}
%:
	\begin{cor}\label{cor:Orthogonality}
		For each $r \geq 2$, a generic stable vector bundle $ G \in \mathcal U_C(r, 2r)$ satisfies 
		\[
			\Ext^p_C( \mathcal R,  G) = 0,\ p=0,1.
		\]
	\end{cor}
	\begin{proof}
		Assume that $ G_i \in \mathcal U_C(r_i,2r_i)$\,($i=1,2$) are stable vector bundles satisfying $\Ext^p_C(\mathcal R,  G_i) = 0$. Then $ G_3:=  G_1 \oplus  G_2$ is a semistable vector bundle satisfying $\Ext^p_C(\mathcal R,  G_3)=0$. By the semicontinuity, we see that $\Ext^p_C(\mathcal R,  G)=0$ for a general $ G \in \mathcal U_C(r_1+r_2,\,2(r_1+r_2))$. By \cite[Proposition~9]{Beauville2016:IntroductionUlrich}, Proposition~\ref{prop:UlrichViaDerivedCategory}, and Proposition~\ref{prop: Generic orthogonality r=3}, there are  vector bundles $ G_1 \in \mathcal U_C(2,4)$ and $ G_2 \in \mathcal U_C(3,6)$ such that $\Ext^p_C(\mathcal R,  G_i)=0$. Since direct sums of $ G_1$ and $ G_2$ can produce all the ranks${}\geq 4$, we get the desired result.
	\end{proof}
%:	
%:

Recall that the projection image $ F = \Phi_{ \mathcal S}^{!} (\mathcal E^* (1))$ is always a semistable vector bundle. It is easy to see that both the stability and the strict semistability are preserved by this Fourier-Mukai projection.
%%%

\begin{prop}\label{prop: Stability comparison}
Let $\mathcal E$ be an Ulrich vector bundle of rank $r \geq 2$, and let $ F := \Phi_\mathcal S^! \mathcal E^*(1)$ be a semistable vector bundle on $C$. If $\mathcal E$ is stable (resp. strictly semistable), then so is $ F$.
\end{prop}
\begin{proof}
	First assume that $\mathcal E$ is strictly semistable. There is a destabilizing sequence
	\[
	0 \to \mathcal E' \to \mathcal E \to \mathcal E'' \to 0
	\]
	where $\mathcal E'$ and $\mathcal E''$ are Ulrich bundle of smaller ranks by Proposition~\ref{prop: semistability of Ulrich}. This gives the following short exact sequence 
	\[
	0 \to  F'' := \Phi_\mathcal S^! \mathcal E''^*(1) \to  F = \Phi_\mathcal S^! \mathcal E^*(1) \to  F' := \Phi_\mathcal S^! \mathcal E'^*(1) \to 0.
	\]
	Since $\mathcal E''$ is Ulrich, we see that $ F'' \subset  F$ is a vector bundle of slope 2 on $C$, so $ F$ cannot be stable.
	
	Now assume that $\mathcal E$ is stable, but $F$ is strictly semistable. Consider the destabilizing sequence
	\[
		0 \to  F'' \to  F \to  F' \to 0.
	\]
	Since $ F$ comes from an Ulrich bundle, the conditions in Proposition~\ref{prop:UlrichViaDerivedCategory} ensures that $h^1( F' \otimes \mathcal S_x)=0$ and $\Ext^p_C(\mathcal R,\,  F'^* \otimes \omega_C^{\otimes 2}) =0$. It follows that $\mathcal E'$ is Ulrich where $\mathcal E'^*(1) := \Phi_\mathcal S( F')$. The existence of the nonzero map $\mathcal E^*(1) \to \mathcal E'^*(1)$ leads to a contradiction; indeed, $\mathcal E^*(1)$ is stable of $\mu=0$ and $\mathcal E'^*(1)$ is semistable of $\mu=0$, thus there is no nonzero map from $\mathcal E^*(1)$ to $\mathcal E'^*(1)$.
\end{proof}
%:

%:
	To sum up the above discussions, we have the following theorem.
	\begin{thm}\label{thm: Main thm}
		Let $\mathcal M(r)$\,($r \geq 2$) be the moduli space of S-equivalence classes of Ulrich bundles of rank $r$ over $X$. The projection functor $\Phi_\mathcal S^! \colon \D(X) \to \D(C)$ induces the morphism 
		\[
			\varphi \colon \mathcal M(r) \to \mathcal U_C(r,2r),\ [\mathcal E] \mapsto \varphi(\mathcal E):=[\Phi_\mathcal S^!( \mathcal E^*(1))]
		\]
		of moduli spaces. Moreover, $\varphi$ satisfies the following properties:
		\begin{enumerate}
			\item set-theoretically, $\varphi$ is an injection;
			\item $\varphi$ maps stable(resp. semistable) objects to stable(resp. semistable) objects;
			\item let $\mathcal M^{\sf s}(r)$ be the stable locus. Then $\varphi$ induces an isomorphism of $\mathcal M^{\sf s}(r)$ onto
			\[
				\varphi( \mathcal M^{\sf s}(r)) = \left\{  [ F] \in \mathcal U_C^{\sf s}(r,2r)  : 
				\begin{array}{ll}
					\Ext^p_C(\mathcal R,\,  F^* \otimes \omega_C^{\otimes 2})=0,\ p=0,1,\\
					h^1( F \otimes \mathcal S_x)=0\ \text{for each }x \in X.
				\end{array}
				\right\},
			\]
			which is a nonempty open subscheme of $\mathcal U_C(r,2r)$.
		\end{enumerate}
	\end{thm}
	\begin{proof}
		First of all, to be well defined, $\varphi$ has to preserve S-equivalence classes. Assume that $\mathcal E_1$ and $\mathcal E_2$ are Ulrich bundles which are S-equivalent, \textit{i.e.} there are Jordan-H\"older filtrations
		\[
			0 =  \mathcal E_i^{(0)} \subset  \mathcal E_i^{(1)} \subset \ldots \subset  \mathcal E_i^{(m)} = \mathcal E_i^*(1)
		\]
		such that $ \mathcal E_1^{(j)} /  \mathcal E_1^{(j-1)} =: \op{gr}_j(\mathcal E_1^*(1)) \simeq \op{gr}_j (\mathcal E_2^*(1)) :=  \mathcal E_2^{(j)} /  \mathcal E_2^{(j-1)}$. For each $j$,
		\[
			0 \to  \mathcal E_i^{(j-1)} \to  \mathcal E_i^{(j)} \to \op{gr}_j (\mathcal E_i^*(1)) \to 0
		\]
		is a short exact sequence of Ulrich bundles by Proposition~\ref{prop: semistability of Ulrich}. The map $\varphi$ preserves both the stability and the strict semistability by Proposition~\ref{prop: Stability comparison}, so it immediately follows that
		\[
			0 = \Phi_\mathcal S^!( \mathcal E_i^{(0)}) \subset \Phi_\mathcal S^!( \mathcal E_i^{(1)}) \subset \ldots \subset \Phi_\mathcal S^!( \mathcal E_i^{(m)}) = \varphi(\mathcal E_i)
		\]
		is a Jordan-H\"older filtration with $\op{gr}_j( \varphi(\mathcal E_i)) \simeq \Phi_\mathcal S^!( \op{gr}_j( \mathcal E_i^*(1)))$. This shows that $\varphi(\mathcal E_1)$ and $\varphi(\mathcal E_2)$ are S-equivalent.
		
		 The statement (1) follows from the fact that $\Phi_\mathcal S \colon \D(C) \to \Phi_\mathcal S(\D(C))$ is an equivalence of categories, and $\mathcal E^*(1) \in \Phi_\mathcal S(\D(C))$ for each Ulrich bundle $\mathcal E$ over $X$.	The statement (2) is already proved in Proposition~\ref{prop: Stability comparison}, so it only remains to prove (3). For any stable Ulrich bundle $[\mathcal E] \in \mathcal M^{\sf s}(r)$, the functor $\Phi_\mathcal S^!$ induces
		\begin{align*}
			T_{[\mathcal E]}\mathcal M^{\sf s}(r) 
			&\simeq \Ext^1_X(\mathcal E,\mathcal E)  \\
%			&\simeq \Ext^1_X(\mathcal E^*(1),\mathcal E^*(1)) \\
			&\simeq \Ext^1_C(\varphi(\mathcal E),\varphi(\mathcal E)) \\
			&\simeq T_{[\varphi(E)]} \mathcal U_C^{\sf s}(r,2r).
		\end{align*}
		Hence together with (1), $\varphi$ is an isomorphism near $[\mathcal E]$. Finally, by Proposition~\ref{prop: generic FM image is locally free} and Corollary~\ref{cor:Orthogonality}, $\varphi(\mathcal M^{\sf s}(r))$ is open and nonempty.
	\end{proof}
%:
	\begin{rmk}
		It is not true in general that $\varphi(\mathcal M^{\sf s}(r)) = \mathcal U_C^{\sf s}(r,2r)$. For example, choose a point $P \in C$ and consider a stable bundle $ F := \mathcal R^* \otimes \mathcal O_C(-P) \otimes \omega_C^{\otimes 2}$ of rank $4$ and degree $8$. Then $ F^* \otimes \omega_C^{\otimes 2} = \mathcal R \otimes \mathcal O_C(P)$, hence we see that
		\[
			\Hom_C(\mathcal R,\,  F^* \otimes \omega_C^{\otimes 2}) \neq 0.	
		\]
		This shows that $\varphi( \mathcal M^{\sf s}(4))$ is a proper subset of $\mathcal U_C^{\sf s}(4,8)$.
	\end{rmk}
%:
%:
%:
%:
	The relation between jumping lines and instanton bundles has been studied in \cite{Kuznetsov:Instanton}. For stable Ulrich bundles, we show that a generic line is not jumping. Recall that $\ell \subset X$ is a \emph{jumping line} for $\mathcal E$ if the direct sum decomposition of $\mathcal E\big\vert_\ell$ contains at least two non-isomorphic direct summands.
%:
	\begin{prop}
		Let $\mathcal E$ be a stable Ulrich bundle of rank $r$ over $X$. For a generic line $\ell \subset X$, $\mathcal E\big\vert_\ell \simeq \mathcal O_X(1)^{\oplus r}$.
	\end{prop}
	\begin{proof}
		We may assume that $\xi = \mathcal O_C(P)$ for a point $P \in C$. Indeed, if we choose a suitable $ L \in \Pic^0 (C)$ and make a replacement $\mathcal S' := \mathcal S \otimes p_C^*  L$, then all the arguments in this section are still valid for the new Fourier-Mukai transform $\Phi_{\mathcal S'} \colon \D(C) \to \D(X)$ and its right adjoint $\Phi_{\mathcal S'}^!$. In particular, the Raynaud bundle $\mathcal R'$ obtained from $\Phi_{\mathcal S'}^! \mathcal O_X(-2)$ as in Lemma~\ref{lem: Projection image and Raynaud bundle} satisfies $\mathcal R' = \mathcal R \otimes  L$.
		
		Let $ F:= \Phi_\mathcal S^!( \mathcal E^*(1))$ and $ G :=  F^* \otimes \omega_C^{\otimes 2}$. We have $ G \otimes \xi^* \neq \mathcal R$; otherwise
		\[
			0 = \Hom_C(\mathcal R,\,  G) = \Hom_C(\mathcal R,\, \mathcal R \otimes \mathcal O_C(P)) \neq 0
		\]
		gives a contradiction. Since $ G$ is stable and $ G \otimes \xi^* \neq \mathcal R$, we have $\Hom_C(\mathcal R,\,  G \otimes \xi^*)=0$. By \cite[Lemma~2.4 and Theorem~2.5]{Hein:Raynaud}, $H^0(C,\, L\otimes  G \otimes \xi^* ) =0$ for a general $ L \in \Pic^0 (C)$. On the other hand,
		\begin{align*}
			H^p(C,\,  L \otimes  G \otimes \xi^*) 
			&= \Ext^{1-p}_C(  G,\,  L^* \otimes \xi \otimes \omega_C)^* \\
			&= \Ext^{1-p}_C(  L \otimes \xi^* \otimes \omega_C , \,  F )^* \\
			&= \Ext^{1-p}_X( \Phi_\mathcal S(  L \otimes \xi^* \otimes \omega_C) ,\, \mathcal E^*(1))^*.
		\end{align*}
		We have $\Phi_\mathcal S(  L \otimes \xi^* \otimes \omega_C) = \mathcal I_\ell(1)[-1]$ for a line $\ell \subset X$ and its ideal sheaf $\mathcal I_\ell$\,(\textit{cf.} \cite[Lemma~5.5]{Kuznetsov:Instanton}). This establishes a bijection between $\Pic^1 (C)$ and the Fano variety $F(X)$ of lines in $X$. Thus, 
		\[
			H^p(C,\,  L \otimes  G \otimes \xi^*) \simeq \Ext_X^{2-p}( \mathcal I_\ell,\, \mathcal E^*)^* \simeq H^{p+1}( \mathcal E \otimes \mathcal I_\ell (-2) ).
		\]
		In the short exact sequence $0 \to \mathcal E \otimes \mathcal I_\ell(-2) \to \mathcal E(-2) \to \mathcal E(-2) \otimes \mathcal O_\ell \to 0$, we easily find that $H^{p+1}( \mathcal E \otimes \mathcal I_\ell(-2)) \simeq H^p( \mathcal E(-2) \otimes \mathcal O_\ell)$. In particular, $h^p(\mathcal E(-2) \otimes \mathcal O_\ell)=0$ which implies $\mathcal E \big\vert_\ell \simeq \mathcal O_X(1)^{\oplus r}$ for a general $\ell \in F(X)$.
	\end{proof}
%:
%%%

We finish this paper by some important remarks. 

\begin{rmk}[Arrondo--Costa revisited]\ 
\begin{enumerate}
	\item Arrondo--Costa's classification\,(Theorem~\ref{thm: Arrondo-Costa ACM rk 2}) also can be interpreted via derived categories of coherent sheaves on $X$. The moduli space of rank $2$ ACM bundles of line type is isomorphic to the abelian surface $J(C)$, and the interpretation in terms of categorical language is explained in \cite[Lemma~5.5]{Kuznetsov:Instanton}. The moduli space of rank $2$ ACM bundles of conic type is isomorphic to $C$ and this can be explained by the result of \cite{BondalOrlov:SODforAlgVar} because the image of a conic type ACM bundle along the projection functor is a skyscraper sheaf. Finally, rank $2$ ACM bundles of elliptic curve type are Ulrich, hence $\mathcal E \mapsto \Phi_\mathcal S^!( \mathcal E^*(1))$ shows that the moduli space of ACM bundles of elliptic curve type is isomorphic to an open subset of $\mathcal U_C(2,4)$.
	\item  We observed above that the rank 3 vector bundle $\mathcal E$ constructed in \cite[Example 4.4]{ArrondoCosta2000} is not Ulrich. Indeed, two global sections of $\omega_D(-1)$ has a nontrivial linear relation, that is,
\[
H^1(\mathcal E^* (1)) \simeq \ker [ H^0(\omega_D(-1)) \otimes H^0(\mathbb P^5, \mathcal O_{\mathbb P^5}(1)) \to H^0(\omega_D)] \simeq \C^1.
\]
Hence $h^2(\mathcal E(-3)) = h^3 (\mathcal E(-3)) = 1$. Nevertheless, it is still a very interesting vector bundle as in the following sense. Since $\mathcal E(-1)$ and $\mathcal E(-2)$ have no cohomology, we see that $\mathcal E^* (1)$ is a semistable vector bundle of rank 3 contained in $\Phi_\mathcal S \D(C)$. Indeed, the nonzero section of $H^0 (\mathcal E^*(1)) \simeq H^3 (\mathcal E(-3))^*$ induces a short exact sequence
\[
0 \to \bar{ \mathcal E} \to \mathcal E(-1) \to \mathcal O_X \to 0,
\]
where $\bar{\mathcal E}$ is a rank 2 vector bundle so called an ``instanton bundle'' of charge 3 (see \cite{Faenzi:Instanton} and \cite[Definition 1.1 and Theorem 3.10]{Kuznetsov:Instanton}). Note that rank 2 Ulrich bundles are instanton bundles of charge 2, which are minimal. Arrondo--Costa construction shows the existence of a non-minimal instanton bundle.
%:
\end{enumerate}
\end{rmk}

\begin{rmk}
The second Raynaud bundle $\mathcal R$ has an interesting property. Note that a (semi-)stable vector bundle $ F$ of rank r and slope $g-1(= 1)$ on $C$ defines the theta locus
\[
\Theta_{ F} := \{  L \in \Pic^0 (C) \ | \ h^0 (C,  F \otimes  L) \neq 0 \},
\]
which is a natural generalization of the theta divisor. The locus is either a divisor linearly equivalent to $r \Theta$ where $\Theta \subset \Pic^0(C)$ is the usual theta divisor, or the whole Picard group $ \Pic^0(C)$. Indeed, the theta map
\[
\theta : \mathcal {SU}_C (r, \det F) \dashrightarrow \lvert r \Theta \rvert
\]
gives a rational map, which is a morphism when $r \le 3$. However, when $r=4$, $\mathcal R$ does not have a theta divisor since $h^0(\mathcal R \otimes L ) = 1$ for every $L \in \Pic^0(C)$ as treated above (see also \cite[Lemma 5.9]{Kuznetsov:Instanton}). We refer interested readers to \cite{Hein:Raynaud,Pauly:Raynaud, Popa:GeneralizedTheta} for more details on generalized theta divisors and $\mathcal R$.

%:
The strange duality provides a following geometric interpretation in terms of generalized theta divisors. Denote $\mathcal L$ by the ample generator of $\Pic \mathcal {SU}_C (4, \det \mathcal R)$, we see that $\mathcal R$ is a base point of $|\mathcal L^k|$ if and only if
\[
	H^0(C, \mathcal R \otimes  G) \neq 0, \text{ for all }  G \in \mathcal U_C(k,0). 
\] 
By Serre duality, the above condition is equivalent to
	\[
		\Hom_C(\mathcal R,  G^* \otimes \omega_C) = \Ext^1_C (\mathcal R,  G^* \otimes \omega_C) \neq 0.
	\]
	Note that $ G^* \otimes \omega_C$ is a vector bundle of rank $k$ and degree $2k$.
	
Corollary \ref{cor:Orthogonality} actually implies that $\mathcal R$ is not a base point of $|\mathcal L^k|$ for $k \ge 2$. Since Proposition \ref{prop: Generic orthogonality r=3} holds not only for $\mathcal R$ but for any stable rank 4 vector bundle of degree 4, we conclude that
\begin{enumerate}
\item $\mathcal R \not \in Bs |\mathcal L^2|$, {\textit{i.e.}}, $Bs | \mathcal L^2|$ is a proper subset of $Bs |\mathcal L| = \{ \text{the set of 16 Raynaud type bundles on } C \}$ which correspond to 16 theta characteristics of $C$;
\item The linear system $|\mathcal L^k|$ is base-point-free for $k = 3$.
\end{enumerate}
Since $|\mathcal L^k|$ is base-point-free for $k \ge 4$ \cite[Theorem 8.1]{PopaRoth}, the above statement answers to the question by Popa and Roth for $g=2$ and $r=4$ (cf. \cite[Section 8]{PopaRoth}).

 Even though our argument do not assure that a generic vector bundle $ F \in \mathcal U_C(2,4)$ is orthogonal to all the 16 Raynaud type bundles, however, it sounds very promising that $|\mathcal L^2|$ is also base-point-free.
\end{rmk}

\begin{rmk}
The strategy in Proposition \ref{prop:UlrichViaDerivedCategory} is also useful to classify Ulrich bundles for smooth complete intersection varieties of two even dimensional quadrics of higher dimensions. In higher dimensional cases, we also observe that every Ulrich bundle is a image of Fourier-Mukai transform of a semistable vector bundle on the associated hyperelliptic curve from Bondal-Orlov's semiorthogonal decomposition. Moreover, the moduli space of stable Ulrich bundles is a smooth Zariski open subset of the moduli space of stable vector bundles on the associated hyperelliptic curve. However, showing the existence becomes more complicated for higher dimensional cases. For instance, there is no Ulrich bundle of rank 2 on such an $n$-dimensional del Pezzo variety of degree 4 when $n \ge 5$ \cite[Theorem 6.3]{Casnati:rank2ACM}.  By the way, we know the existence of Ulrich bundles of certain rank in these cases from \cite{BuchweitzEisenbudHerzog}. Therefore, it is interesting to compute all the possible ranks of Ulrich bundles on the higher dimensional smooth complete intersection varieties of two even dimensional quadrics. For example, \cite[Theorem 8.1]{PopaRoth} enables us to make a wild expectation that an Ulrich bundle of rank $2^{2g-2}$ might exist.
\end{rmk}

\begin{acknowledgement}
The authors thank Fabrizio Catanese and Universit{\"a}t Bayreuth for kind hospitality during their visit. Yonghwa Cho would like to express his gratitude to JongHae Keum and Korea Institute for Advanced Study for hospitality when he was visiting there. Yeongrak Kim thanks George Harry Hitching, Mihnea Popa, and Frank-Olaf Schreyer for helpful discussion and suggestions. Kyoung-Seog Lee is grateful to Mudumbai Seshachalu Narasimhan for many invaluable teachings, encouragements and kind hospitality. He thanks Alexander Kuznetsov, Carlo Madonna, and Paolo Stellari for kind explanations and motivating discussions. Part of this work was done while he was a research fellow of Korea Institute for Advanced Study and was visiting Indian Institute of Science. He thanks Korea Institute for Advanced Study and Indian Institute of Science for wonderful working conditions and kind hospitality. He thanks Gadadhar Misra for kind hospitality during his stay in Indian Institute of Science.

	Yonghwa Cho was partially supported by Basic Science Research Program through the NRF of Korea funded by the Ministry of Education (2016930170). Yeongrak Kim was supported by Basic Science Research Program through the National Research Foundation of Korea funded by the Ministry of Education (NRF-2016R1A6A3A03008745). Kyoung-Seog Lee was supported by IBS-R003-Y1.
\end{acknowledgement}

%%%
%:

\bibliography{CKL_Bib}
%:
%:
%:
%:
%:
%:
%:
\end{document}